\documentclass[reqno]{amsart}
\usepackage{url}
\usepackage{amssymb,amsthm,amsfonts,amstext}
\usepackage{amsmath}
\usepackage{mathptmx}       
\usepackage{helvet}         
\usepackage{courier}        
\usepackage{graphicx}
\usepackage{enumerate}
\usepackage{a4wide}
\usepackage[active]{srcltx}
\usepackage[brazil,english]{babel}
\usepackage[latin1]{inputenc}
\usepackage{color}

\newtheorem{theorem}{Theorem}

\newtheorem{corollary}[theorem]{Corollary}

\newtheorem{definition}[theorem]{Definition}
\newtheorem{example}[theorem]{Example}

\newtheorem{lemma}[theorem]{Lemma}

\newtheorem{proposition}[theorem]{Proposition}

 \newcommand{\mc}[1]{{\mathcal #1}}

 \newcommand{\bb}[1]{{\mathbb #1}}

\newcommand{\<}{\langle}
 \renewcommand{\>}{\rangle}
\newcommand{\pfrac}[2]{\genfrac{}{}{}{1}{#1}{#2}}

\newcommand{\at}[2]{\genfrac{}{}{0pt}{}{#1}{#2}}

\begin{document}

\title[A thermodynamic formalism for continuous time Markov chains]{A thermodynamic formalism for continuous time Markov chains with values on the Bernoulli Space: entropy, pressure and large deviations}

\author{Artur Lopes}
\address{UFRGS, Instituto de Matem\'atica, Av. Bento Gon\c calves, 9500. CEP 91509-900, Porto Alegre, Brasil}
\curraddr{}
\email{arturoscar.lopes@gmail.com}
\thanks{}

\author{Adriana Neumann}
\address{UFRGS, Instituto de Matem\'atica, Av. Bento Gon\c calves, 9500. CEP 91509-900, Porto Alegre, Brasil}
\curraddr{}
\email{aneumann@impa.br}
\thanks{}

\author{Philippe Thieullen}
\address{Institut de Math\'ematiques, Universit\'e Bordeaux 1, Bourdeaux, France}
\curraddr{}
\email{philippe.thieullen@math.u-bordeaux1.fr}
\thanks{}

\date{\today}
\maketitle

\begin{abstract}

Through this paper we analyze the ergodic properties of continuous time Markov chains with values on the one-dimensional spin lattice
$\{1,\dots,d\}^{\bb N}$ (also known as the Bernoulli space). Initially, we consider as the infinitesimal generator the operator $L={\mc L}_A -I$,
where ${\mc L}_A$ is a discrete time Ruelle operator (transfer operator), and $A:\{1,\dots,d\}^{\bb N}\to \mathbb{R}$ is a given fixed Lipschitz function.
The associated continuous time stationary Markov chain will define the\emph{ a priori }probability.

Given a Lipschitz interaction $V:\{1,\dots,d\}^{\bb N}\to \mathbb{R}$, we are interested in Gibbs (equilibrium) state for such $V$.
This will be  another continuous time stationary Markov chain. In order to analyze this problem we will use  a continuous time Ruelle operator (transfer operator)
naturally associated to $V$.
Among other things we  will show that  a continuous time Perron-Frobenius Theorem is true in the case $V$ is a   Lipschitz function.

We also introduce an entropy, which is negative (see also \cite{LMMS}), and we consider a variational principle of pressure. Finally, we
analyze  large deviations properties for the empirical measure in the continuous time setting using results by Y. Kifer (see \cite{Ki1}).
In the last appendix of the paper we explain why the techniques we develop here have the capability to be applied to the analysis of
convergence of a certain version of the Metropolis algorithm.

\end{abstract}

\bigskip

\section{Introduction}

In this paper we will consider thermodynamic formalism in a continuous time setting in a similar way as in  \cite{BCLMS} and \cite{LMMS},
where the time is discrete.
In order to be able to work in this new context (continuous time) we need to consider first a stationary continuous time Markov
chain, and this will define the \emph{a priori probability}, on the space of trajectories. The infinitesimal generator of this
continuous time  Markov chain  will be associated to a  discrete time Ruelle operator. Namely, we consider as the infinitesimal generator the operator $L={\mc L}_A -I$,
where ${\mc L}_A$ is a discrete time Ruelle operator.

In the continuous time setting
we will  be able to define a new Ruelle operator, in a similar fashion as in  \cite{LMMS}.
The continuous time setting  requires some extra effort to get results, as can be seen in \cite{BEL} and \cite{kl}.
However, we will be able to get here the analogous  properties of the Ruelle operator which appear in the discrete time setting (transfer operators).
Based on the theory of stochastic processes we can define the  continuous time Ruelle operator, as well as the
entropy and the pressure in this new context.

The Heat-Bath Glauber dynamics is a continuous time
Markov chain as  described in \cite{BKMP}. Questions related to the Ising model on a regular tree are consider in
this mentioned work. The infinitesimal generator we consider here is a generalization of (1) in this paper.
Our setting is a general one where several possible models of Statistical Mechanics can fit well (see for instance \cite{SZ}).

In a future work we will apply the techniques we developed here to the analysis of a special version of the Metropolis algorithm (see
\cite{RT}, \cite{DS1}, \cite{DS2} and \cite{Leb})
which will be  suitable for applications in problems where the state space is the one-dimensional spin lattice.
Suppose $A$ is fixed for good
(in this way we fix an a priori probability). Given a certain function
$V: \{1,\dots,d\}^{\bb N}\to \mathbb{R}$  we would like to find
the point $x \in \{1,\dots,d\}^{\bb N}$ which maximize this function. For each value $\beta>0$ one can consider the potential $\beta V$
and the associated Gibbs state $\bb P^{\beta \, V}$ which is a probability over the set of  continuous time  paths (a new continuous time Markov chain). Now, from ergodicity,
if we choose at random a continuous time
sample path we get a good approximation for the occupation time  probability on  $\{1,\dots,d\}^{\bb N}$ (Monte-Carlo method). This path can be
seen as a random  algorithm which is
exploring the configuration space $\{1,\dots,d\}^{\bb N}.$  In Appendix F we show that if we take $\beta$ more and more large,
then, the sample path we choose will stay more an more time close to the maximimum of $V$. For  large and fixed $\beta$ it is
important, from the point of view of the algorithm, to understand the large deviation properties of the associated empirical probability
of the path on $\{1,\dots,d\}^{\bb N}.$ This is related to the second part of our paper. This will be carefully explained in the end of  Appendix F.

We point out that some of the results we obtain in our  paper are due to the good properties
already known for the classical  Ruelle operators ${\mc L}_A$ on discrete time (transfer operators).
 So we begin by recalling some important topics of this subject.

Consider the shift $\sigma$ acting on the one-dimensional spin lattice $\{1,\dots,d\}^{\bb N}$.
We denote by $P(B)$ the pressure of the potential $B: \{1,\dots,d\}^{\bb N}\to \mathbb{R}$ (see \cite{Cra}, \cite{LoT} and \cite{PP}).
The value $P(B)$ is the supremum of $h(\mu)+ \int B \mbox{d} \mu$, among all $\sigma$-invariant probabilities on $\{1,\dots,d\}^{\bb N}$,
where $h(\mu)$ is the Kolmogorov entropy of the invariant probability $\mu$.
If $B$ is Lipschitz there exists a unique $\mu_B$ such that
$P(B)= h(\mu_B) + \int B \mbox{d} \mu_B.$ We call $\mu_B$  the (discrete time)  equilibrium state for $B$ (see \cite{LoT} and \cite{PP}).
Each point $x\in \{1,\dots,d\}^{\bb N}$ has a finite number of preimages $y\in \{1,\dots,d\}^{\bb N}$ by $\sigma$.
For  a  Lipschitz potential $B$ we define the  Ruelle operator by
\begin{equation*}
 {\mc L}_B(f) \, (x) =\sum_{\sigma(y) =x} \, e^{B(y)}
\,f(y)\,,
\end{equation*}
for any continuous function $f:\{1,\dots,d\}^{\bb N}\to\bb R$ and $x\in\{1,\dots,d\}^{\bb N}$.
We say a   Lipschitz potential $A:\{1,\dots,d\}^{\bb N}\to\bb R$ is normalized if for any $x\in \{1,\dots,d\}^{\bb N}$ we have
\begin{equation*}
\sum_{\sigma(y)=x} e^{A(y)}=1\,.
\end{equation*}
To assume that all the potentials which we consider are Lipschitz is an essential issue (but, it could be relaxed to Holder).
Nice references in thermodynamic formalism are \cite{Bal} and \cite{Rue}.

The dual of ${\mc L}_A$ is the operator ${\mc L}_A^*$, which acts on probabilities on $\{1,\dots,d\}^{\bb N}$ in the following way:
\begin{equation*}
\int g\,  \,\mbox{d}{\mc L}_A^*(\nu)= \int {\mc L}_A(g)\, \mbox{d} \nu\,,
\end{equation*}
 for any continuous function $g:\{1,\dots,d\}^{\bb N}\to \mathbb{R}$.
The probability  $\nu$ such that ${\mc L}_A^*(\nu)=\nu$ is called the (discrete time) Gibbs probability.
If $A$ is a Lipschitz  normalized potential, we have $P(A)=0$, and,
one can show that   ${\mc L}_A^*(\mu_A)= \mu_A$. There is a unique fixed point probability  for ${\mc L}_A^*$.
In this case the Gibbs state for $A$ is the equilibrium state for $A$ (see \cite{PP}).
Equilibrium states describe the probabilities that
 naturally appear in problems in Statistical Mechanics over the one-dimensional lattice $\{1,\dots,d\}^{\bb N}$.

After this brief introduction on discrete time dynamics,
we consider now the setting in which we will get our  main results.
Let $\mc D:=\mc D\big([0,+\infty), \{1,\dots,d\}^{\bb N}\big)$ be the path space of
\emph{c\`adl\`ag} (right continuous with left limits) trajectories  taking values in $\{1,\dots,d\}^{\bb N}$ (see \cite{Li} and \cite{Pro}).
This space is usually endowed with the Skorohod metric (for more details about this metric see \cite{EK}), and it is called  the Skorohod space. A typical element of $\mc D$ is a function $\omega:[0,\infty)\to \{1,\dots,d\}^{\bb N}$
 which is right continuous and has left limit in all points.
This space is
complete and has a countable dense  set, in other words, it is a Polish space, but it is not compact (see \cite{EK}).
The continuous time dynamics that we consider here will be given by the action of
the continuous time shift $\Theta_t: \mc D\to \mc D, t\geq 0$. Given $t_0>0$
and a path $\omega\in \mc D$ on the Skorohod space, then, $\Theta_{t_0}(\omega)$
is the path $\eta$ such that $\eta(t)= \omega(t+t_0)$, for all $t\geq 0$. We consider here the dynamics associated to such  semiflow,
$\{\Theta_t, \,t\geq 0\}$. Notice that the transformation $\Theta_t$ is not injective, because for a fixed $t$ and
for each $\eta\in \mc D$ there exists an uncountable number of preimages $\omega\in \mc D$ such that $\Theta_t(\omega)=\eta$.

We said that the probability $\tilde{\bb P}$ on the Skorohod space is invariant if it is invariant for the semiflow
$\{\Theta_t, t\geq 0\}$; that is, for any Borel set $\mc K$  in $\mc D$ and $t>0$, we have
$\tilde{\bb P}[ \Theta_t^{-1} (\mc K)]= \tilde{\bb P}[\mc K]$.
In order to find  invariant  probabilities on the Skorohod space,
it is natural  to consider a continuous time Markov chain taking values on
the one-dimensional spin lattice (we point out that not all invariant  probabilities on the Skorohod space  appear on this way).
In this direction, we will use a Ruelle operator (transfer operator) with Lipschitz normalized potencial $A:\{1,\dots,d\}^{\bb N}\to \mathbb{R}$ for defining
the infinitesimal generator of a  continuous time
Markov chain in the form
 \begin{equation*}
({\mc L}_A -I)(f)  (x)\,= \,\sum_{\sigma(y) =x} \, e^{A(y)} \,[f(y)-f(x)] \,,
 \end{equation*}
for all bounded measurable function $f:\{1,\dots,d\}^{\bb N}\to\bb R$ and $x\in\{1,\dots,d\}^{\bb N}$.

Denote by $L:={\mc L}_A -I$ this infinitesimal generator.
For $x\in\{1,\dots,d\}^{\bb N}$, consider an initial probability measure $\delta_x$ on $\{1,\dots,d\}^{\bb N}$, and
denote by $\bb P_{x}$ the
probability measure on
$\mc D$,
which is induced by the infinitesimal generator $L$ and the initial probability $\delta_x$.
It defines a  Markov process
$\{X_t;\, t\geq 0\}$ with values on the state space
$\{1,\dots,d\}^{\bb N}$ (see \cite{DS}, \cite{DL},  and \cite{kl}).
As usual, when necessary, we will consider the canonical version of the process,
i.e., $X_s(\omega)=\omega(s):=\omega_s$, for any $\omega \in \mc D$ and $s\geq 0$.
 The stochastic semigroup generated by $L$ is $\{P_t:=e^{tL},\, t\geq 0\}$
(the operator $L$ is bounded and $L(1)\equiv 0$).
The expectation concerning $\bb P_{x}$ is denoted by $\bb
E_{x}$.
Given $\mu$ an initial probability on $\{1,\dots,d\}^{\bb N}$, we can define the probability $\bb P_\mu$ on $\mc D$ as
$$\bb P_\mu[\mc K]=\int_{\{1,\dots,d\}^{\bb N}}\bb P_{x}[\mc K]\,\mbox{d}\mu(x)\,,$$
for all Borel set $\mc K\subset\mc D$.

The above process describes the behavior
of  a particle, such that when  located at $x\in \{1,\dots,d\}^{\bb N}$,
 jumps to one of its $\sigma$-preimages $y$, with probabilities described by
 $e^{A(y)}$ and after an exponential time of parameter $1$.
Notice that for almost every trajectory $\omega$ beginning in
$x=\omega_0$, all  the values $\omega_t, t\geq 0$, which are possibly
attained belong to the {\it total pre-orbit set}, by the shift $\sigma$, of the initial point $x$,
that is, the set of $y$ {\it  such that for some $n\in \mathbb{N}$ we have $\sigma^n (y)=x$}.
The space $\{1,\dots,d\}^{\bb N}$ is not countable. We point out that in most of the papers in the literature the state space  is finite (or, countable).
In this last situation the infinitesimal generator is  a matrix which satisfies the
condition of line sum zero. Here this matrix is replaced by an operator described  by
the expression $L={\mc L}_A -I$, where $A$ is normalized.

The discrete Gibbs state probability $\mu_A$ over  $\{1,\dots,d\}^{\bb N}$ (see \cite{PP}) for the potential $A: \{1,\dots,d\}^{\bb N}\to \mathbb{R}$ clearly
satisfies that
\begin{equation*}
\int L(f)\, \mbox{d} \mu_A=0\,,
\end{equation*}
for all $f$ continuous function, where $L={\mc L}_A -I$.
This is the condition for stationarity of the initial probability of the continuous time Markov chain generated by $L$ (see \cite{S}).
Using that $P_t=e^{tL}$,  we get $\int f \mbox{d}\mu_A= \int P_t f \mbox{d}\mu_A$, for all  $f$ and $t\geq 0$.
Therefore, $\mu_A$ is a stationary initial measure for the continuous time Markov chain associated to the stochastic semigroup $\{P_t, t\geq 0\}$.
Notice that there is a unique probability such that $\mathcal{L}_A^*(\mu_A)=\mu_A$ (see \cite{PP}).
This shows that the initial stationary probability for the Markov semigroup $P_t$ is unique.
The associated probability $\bb P_{\mu_{A}}$ on the Skorohod space is invariant for the semiflow $\{\Theta_t, \,t \geq 0\}$.
In this way by taking different potentials $A$ we can get a large number of invariant probabilities for the continuous time  semiflow.
In appendix G we show that the stationary probability $\bb P_{\mu_{A}}$ is ergodic for the continuous time shift $\{\Theta_t, \,t \geq 0\}$.

One can also ask if this
process $\{X_t=X_t^{\mu_A}\!\!,\;\,\, t \geq 0\}$, with initial condition $\mu_A$, is ergodic for the stochastic semigroup, that is,
if the following is true: if for a given measurable $f$ we have that  $L(f)=0$, then,  $f$ is constant $\mu_{A}$ - a.s.
This is indeed the case  and it will be proved in the beginning of next section. We point out that now the meaning of the word ergodic for $\mu_A$ (a probability on the state space $\{1,\dots,d\}^{\bb N}$) is for the continuous time evolution of the stochastic semigroup.

The probability $\bb P_{\mu_{A}}$ induced on $\mc D$ by $L$
and the initial probability $\mu_A$ will be called the \emph{ a priori probability}.
The process $\{X_t=X_t^{\mu_A}\!\!,\;\,\, t \geq 0\}$ is called the {\it{a priori process}}.
We  will  need all of the above in order to define  the continuous time Ruelle operator.

One can ask if $L={\mc L}_A- I$ acting on the Hilbert space $\bb L^2(\mu_A)$ is symmetric.
The answer to this question  is no, because
$L^* = {\mc K} - I,$ where ${\mc K}$ is the Koopman operator,
$g\,\to \,{\mc K}(g) = g \circ \sigma$ (according to \cite{PP}). Therefore, in our setting  the process is not reversible.
In order to make the system reversible we could consider, as usual,
 the generator $\frac{1}{2} (L + L^*).$ For this new process the particle can jump either way:
forward or backward (for the action of  $\sigma$). We briefly consider such process in the end of the paper (see Appendix \ref{end}).

In our reasoning we will consider a fixed choice of $A$ and this defines an\emph{ a priori }probability.
After this is settled,
 we want to analyze the disturbed system by  the intervention of an external  Lipschitz
potential $V: \{1,\dots,d\}^{\bb N}\to \mathbb{R}$.
More precisely, we would like to obtain a new continuous time Markov chain $\{Y^V_T,\,T\geq 0\}$, with state space $\{1,\dots,d\}^{\bb N}$,
which plays the role of
the {\it continuous time} Gibbs state for $V$.
In order to obtain this new process $\{Y^V_T,\,T\geq 0\}$, we need to define the continuous time Ruelle operator acting on
functions defined in the Bernoulli space, based in Feynman-Kac theory (see, for example, \cite{kl} and \cite{S}).
We will also need to show the existence of an eigenfunction $F: \{1,\dots,d\}^{\bb N}\to \mathbb{R}$
 in the case that $V$ is a   Lipschitz function.

We will show that given a Lipschitz potential $V$ there exists   $\lambda=\lambda_V$ and a positive function $F=F_\lambda$ such that for any $T\geq 0$,
$$
e^{T\,(L+V)}(F)= e^{{\lambda_V}  \, T} F\,.$$

\bigskip
One can consider alternatively a continuous time Markov chain associated to a discrete time Ruelle operator
in a more general setting. In fact, this will naturally occur as we will see in the analysis of the  {\it continuous time} Gibbs state for $V$.
When we defined the initial Markov process $\{X_t,\,t\geq 0\}$, we could have chosen another parameter for the exponential clock (not constant equal to $1$).
Below we briefly present how to proceed in this situations.

Let $\gamma $ a continuous positive  function and $B$ a Lipschitz normalized potential,
one could also consider a more general  operator
\begin{equation*}
L_{\gamma,B} (f )(x)= \gamma(x) \sum_{\sigma(y) =x} \, e^{B(y)} \,[f(y)-f(x)] \,,
\end{equation*}
acting on bounded measurable  functions $f:\{1,\dots,d\}^{\bb N}\to \mathbb{R}$. Notice that $L_{\gamma,B}=\gamma\,( {\mc L}_B-
I)$.
We point out that most of the results we will prove in this paper
 are also true if the\emph{ a priori }probability is defined via the stochastic semigroup
$\{e^{\,t L_{\gamma,B}}, \,t\geq 0\}$ (and the associated stationary initial probability), instead of $\{e^{tL},\,t\geq 0\}$.
In this case, if we denote $\mu_{B,\gamma} =\frac{1}{\gamma}\,
\frac{\mu_B}{\int \frac{1}{\gamma}\, \mbox{d} \mu_B},$ then, for any continuous function $f:\{1,\dots,d\}^{\bb N}\to \mathbb{R}$
$$\int L_{\gamma,B}(f) \,\mbox{d} \mu_{B,\gamma}=0\,,$$
where $\mu_B$ is the {\it discrete time} equilibrium state for $B$.
Then, $\mu_{B,\gamma}  $ is the   initial stationary probability for the continuous time Markov process with infinitesimal generator $L_{\gamma,B}$.
It is also stationary for the flow $\{\Theta_t, \,t\geq 0\}$.
Notice that $\mu_{B,\gamma}  $ is not invariant for the discrete time action of the shift $\sigma$.
The probability $\mu_B$ is invariant for the discrete time shift $\sigma$.

We denote by $\{Z_t,\,t\geq 0\}$ the continuous time Markov chain taking
values in the one-dimensional spin lattice $\{1,\dots,d\}^{\bb N}$ generated by  such $L_{\gamma,B}$ and a given initial measure
(not necessarily the process needs to begin on a stationary probability). The process $\{Z_t,\,t\geq 0\}$ with infinitesimal generator
$L_{\gamma,B}$  can be described in the following: if the particle is located at $x\in\{1,\dots,d\}^{\bb N}$,
 then it waits an exponential time
of parameter
$\gamma(x)$, and, then
it jumps to a $\sigma$-preimage $y$ with probability $e^{B (y)}$.
As we will see in the third section of this paper, there exist  $\gamma$ and $B$ which
 naturally appear when we have to  describe properties of what we will call  the
{\it continuous time} Gibbs state for $V$.

\smallskip

Let's come  back to the original setting where the a priory probability on the Skorohod space was defined by
 the process defined by the infinitesimal generator $L= {\mc L}_A- I.$
In order to present in advance the final solution, we can say that the
  {\it continuous time} Gibbs state for $V$  is the process $\{Y^V_T,\,T\geq 0\}$,
which has the infinitesimal generator acting on bounded mensurable functions $f:\{1,\dots,d\}^{\bb N}\to\bb R$ given by
\begin{equation*}
 L^{V}(f)(x)=\gamma_{V}(x)\, \sum_{\sigma(y)=x}e^{B_{V}(y)}\big[f(y)-f(x)\big]\,,
\end{equation*}
where $B_{V}(y):= A(y)- \log \gamma_V(\sigma(y)) +\log F_V(y)-\log F_V(\sigma(y))$,  $\gamma_{V}(x):=1- V(x)+\lambda_{V}$
and the function $F_V$ is such that
$$ \frac{\mc L_A (F_V)(x)}{F_V(x)}=\sum_{\sigma(y)=x} \, \frac{e^{A(y)}\, F_V(y)}{F_V(x)} = 1- V(x)+\lambda_V\,.$$
The appearance of the term $\gamma_V$ in the infinitesimal  generator $L^{V}$ introduce a new element which was not present
in the classical discrete time setting.
This continuous time stationary Markov chain describes the solution one naturally get, from the point of view of Statistical Mechanics,
 for a system under the influence of an external potential $V$.

Now, we can ask: \textquotedblleft Is there a maximizing pressure principle on this setting?"  and
\textquotedblleft Can we talk about entropy in this setting?"  In other words: is this stationary Gibbs probability  an equilibrium measure
in some sense? These questions appear naturally for the discrete time
 Ruelle operator setting (thermodynamic formalism).
Answering these questions is one of the purposes of the present work.
Given an a priory probability (associated to $A$) we will define an entropy for a class of  continuous time Markov chains.
It will be a non-positive number.

Lastly, we study the large deviation principle for the empirical measure associated to the {\it{{\it{\emph{ a priori }process}}}}.
So that one can consider, for each $t\geq 0$ and each $\omega\in\mc D$,   the empirical probability $L_t^\omega$ defined by the
occupational time of the process $\{X_t,\, t\geq 0\}$  on a set, that is, for any Borel $\Gamma\subset\{1,\dots,d\}^{\bb N} $, we have
\begin{equation*}
 \begin{split}
  L_t^\omega (\Gamma)\,=
\, \frac{1}{t} \int_0^t \textbf{1}_{\,\Gamma} (X_s(\omega))\, \mbox{d}s\,.
 \end{split}
\end{equation*}
Then, under ergodicity, we have  $\lim_{t \to \infty} L_t^\omega = \mu_A$, $\bb P_{\mu_A}$-almost surely $\omega$
 (see  page 108 in \cite{S}).
The Ergodic Theorem  says little or nothing about the rate of convergence.
Since $ L_t^\omega$ is random,  it is almost unavoidable to ask oneself about deviations
from the stationary measure $\mu_A$.

Let $\mc M(\{1,\dots,d\}^{\bb N})$ be the set of all measures on $\{1,\dots,d\}^{\bb N}$.
The  large deviation rate function $I:\mc M(\{1,\dots,d\}^{\bb N}) \to \mathbb{R} $, associated to this continuous time process $\{X_t, \,t\geq 0\}$,
helps to estimate the exponential decay of the asymptotic empirical probability of deviations
from the stationary measure $\mu_A$,  when the time parameter $t$  goes to infinity.
Thus, we are naturally led to the investigation and identification
of the large deviations rate function in the set of the measures on Bernoulli space.
We will analyze large deviation properties of the empirical probability (as we mentioned before the system we consider is not reversible).
This is also known as level two large deviation theory (see \cite{DL}, \cite{Ellis} and \cite{Ki1}).
The level one large deviation principle follows by standard procedures: Orey's contraction principle (see for instance \cite{artur1}).

It is important to remark that the understanding of previous results which were obtained for a general potential $V$ plays a fundamental
role in the large deviation  properties of the unperturbed system (with infinitesimal generator $L=\mathcal{L}_A - I)$.
This follows the general philosophy of
\cite{DV1}, \cite{DS} \cite{Ki1} and   \cite{Ki2}.

Suppose $\lambda_V$ is the main eigenvalue we get from the continuous time Ruelle-Perron Operator for $V$. We denote by $\mc C$
the set of continuous functions and by $\mc C^+$ the set of strictly positive continuous functions.

Our main result in the second part of the paper is:
\bigskip

{\bf Theorem A: }  A large deviation principle at  level two for the {\it{a priori process}}  $\{X_t=X_t^{\mu_A}\!\!,\;\,\, t \geq 0\}$ generated by $L= {\mc L}_A- I$
is true with the deviation function $I$ given
$$I(\nu) =\sup_{V\in \mc C}\, \,\Big(\int V \,\mbox{d} \nu\, -\, Q(V)\Big)\,,$$
where $Q(V)$ is  a function which is equal to  the main eigenvalue $\lambda_V$ when $V$ is Lipschitz.

Moreover,

$$
I(\nu) =     - \inf_{u\in \mc C^+}\, \int \frac{L(u)}{u}\, \mbox{d} \nu\,.
$$

\bigskip

We point out that the above Theorem \ref{Ri} in \cite{Ki1} (see also \cite{Ki2})
is presented in a different setting: the state space  is a  Riemannian manifold and it is considered a
certain class of differential operators as infinitesimal generators. We do not consider here such differentiable structure.
\bigskip

The paper is divide in sections as follows: in Section 2, we present the continuous time Ruelle operator and we prove
 the continuous time Perron-Frobenius Theorem. In Section 3,  we present the  continuous time Gibbs state for $V$. This is a continuous time stationary process.
In Section 4, we define relative entropy, pressure and equilibrium state for $V$, and we also prove a variational principle for the Gibbs state.
In Section 5, the main result that we will get is the large deviation principle for the empirical measure associated to the {\it{a priori process}}.
Finally, in the  Appendix we show many technical results using basic tools of continuous time Markov chains.  Among them: we present a Radon-Nikodim derivative result,  we
briefly comment on
the spectrum of ${\mc L}_A - I+V$ on ${\bb L}^2 (\mu)$, where $\mu$ is a natural probability on the Bernoulli space $\{1,\dots,d\}^{\bb N}$, and, finally, some remarks on the associated symmetric process. In this last section we consider a fixed potential $V$ and
we ask about the limit of the invariant probability (invariant for the continuous time equilibrium Gibbs state for $\beta\, V$, when $\beta$ is large) over $\{1,\dots,d\}^{\bb N}$ when temperature goes to zero.

\section{Disturbing the system by an external Lipschitz potential $V$: \\ the continuous time Perron-Frobenius Theorem.}

First of all we recall the definition of the {\it{a priori process}}.
A Lipschitz normalized  potential $A$ will be considered fixed through the whole paper. We denote by $\{P_t,\,t\geq 0\}$,
the stochastic semigroup generated by $L=\mathcal {L}_A-I$.
We need an\emph{ a priori }continuous time  stationary probability for our reasoning,
 for this reason we are considering $\bb P_{\mu_A}$ the probability obtained from the semigroup  $\{P_t$, $t\geq 0\}$ and the initial probability  $\mu_A$.
As we have said, this probability $\bb P_{\mu_{A}}$ plays the role of the\emph{ a priori }measure (see \cite{BCLMS} and \cite{LMMS}).
The associated stochastic process will be denoted by $\{X_t=X_t^{\mu_A}\!\!,\;\,\, t \geq 0\}$.

Given a continuous time stochastic semigroup with compact state space and an initial stationary probability we get a continuous time invariant probability on the Skhorohod space. The continuous time Birkhoff Theorem associated to the continuous time stochastic semigroup (for not necessarily ergodic probabilities) is true (see Remark 1 on page 382 in \cite{Y} or
Theorem 17 page 708 and Exercise 19 page 721 in \cite{DS}).

Therefore, given a continuous function $g:\{1,\dots,d\}^{\bb N}\to \mathbb{R} $ we get an integrable measurable function $f:\{1,\dots,d\}^{\bb N}\to \mathbb{R} $ which describes the possible mean continuous time limits for $g$. This function $f$ is invariant for the action of the stochastic semigroup. Therefore, $L(f)=0.$
In the case $f$ is constant $\mu_{A}$ - a.e.w. then the mean continuous time limits for $g$ are  all the same $\mu_{A}$ - a.e.w. and equal to the $\mu_{A}$ space average on $\{1,\dots,d\}^{\bb N}$.

The probability $\mu_A$ is ergodic for the continuous time action, that is,
the following is true: if for a given $f$ we have that  $L(f)=0$,  then,  $f$ is constant $\mu_{A}$ - a.e.w.
This  follows from the following simple argument suggested by D. Smania:
suppose ${\mc L}_A(f)=f$ for a $\mu_{A}$-integrable $f$, then, for a given $\epsilon$ we can write
$f=g + w$, where $w$ is integrable with $L^1(\mu_A)$ norm smaller than $\epsilon$ and $g$ is Lipchitz. Then,
$f= {\mc L}_A^n(f)= {\mc L}_A^n(g) + {\mc L}_A^n(w)$.

Note that ${\mc L}_A^n(w)$ has $L^1$  norm smaller then $\epsilon$.
Moreover, ${\mc L}_A^n(g)$ converges to a constant $a_g$, where $a_g$ is $\int f d\mu_A$ up to $\epsilon$ (see Theorem 2.2 (iv) \cite{PP}).
Therefore, taking the limit in $n$ we get that $f- a_g $ has norm smaller than $\epsilon$.
Now, taking $\epsilon \to 0$, we get that $a_g$ converges to $ \int f d\mu_A.$
Therefore, for all $x$, $\mu_A$ - a.e.w, we have that
$f(x)= \int f  d\mu_A.$

In the same spirit of  Classical thermodynamic formalism (see \cite{PP}), given a potential $V$ (an interaction),
we want to get here another continuous time Markov process which will be the equilibrium stationary
process for the system under the influence of the potential $V$.

Let $V:\{1,\dots, d\}^{\bb N}\to\bb R$ a  Lipschitz function and consider the operator  $L+V={\mc L}_A -I+V$,
which acts on mensurable and bounded functions $f:\{1,\dots,d\}^{\bb N}\to \bb{R}$ by the expression
\begin{equation*}
(L+V)(f)  (x)\,= \,({\mc L}_A -I)(f)  (x)\,+\,V(x)f(x) \,,
 \end{equation*}
 for all  $x\in\{1,\dots, d\}^{\bb N}$.
 For $T\geq 0$, we consider
\begin{equation}\label{0}
P_{T}^V (f)(x)\,:=\, \bb E_{x} \big[e^{\int_0^{T} V(X_r)\,dr} f(X_T)\big]\,,
\end{equation}
 for all continuous function $f:\{1,\dots, d\}^{\bb N}\to \bb{R}$ and $x\in\{1,\dots, d\}^{\bb N}$. By Feynman-Kac, $\{P_{T}^V,\,T\geq 0\}$
defines a semigroup associated to the infinitesimal operator $L+V={\mc L}_A -I+V$ (see Appendix 1.7 in \cite{kl}).

Let $\mc C$ be the space of continuous functions from $\{1,\dots,d\}^{\bb N}$ to $\bb{R}$ endowed with uniform topology.
Denote by  $\mc C^+$ the subspace of  functions of $\mc C$ which  are strictly positive.
Let $\mc P(\{1,\dots,d\}^{\bb N})$  be the space of probabilities on the Borel sigma-algebra of the one-dimensional spin lattice $\{1,\dots,d\}^{\bb N}$.
Define $\mc M(\{1,\dots,d\}^{\bb N})$ as the space of measures on the Borel sigma-algebra of the Bernoulli space $\{1,\dots,d\}^{\bb N}$.

Notice that, in general, this semigroup in not stochastic, because $P_{T}^V (1)(x)\neq 1$. We want to associate to this semigroup, another one which is also stochastic, this will be only possible due to the next result,
which we consider the main one in this section.

\begin{theorem}[Continuous Time Perron-Frobenius Theorem]\label{prop2} Suppose that $V$ is a  Lipschitz function.
 Then, there exists a strictly positive Lipschitz eigenfunction
 $F:\{1,\dots,d\}^{\bb N}\to (0,+\infty)$
 for the family of operators
$P_{T}^V:\mc C\to\mc C$, $T\geq 0$, associated to an eigenvalue $e^{{\lambda_V}  \, T}$, where  $\lambda=\lambda_V$ depends only on $V$.
By this we mean: for any $T\geq 0$,
\begin{equation*}
P_{T}^V(F)= e^{{\lambda_V}  \, T} F\,.
\end{equation*}
The eigenvalue $\lambda_V$ is simple and it is equal to the spectral radius (maximal).
Moreover, there exists a eigenprobability $\nu_V$ in $\mc P(\{1,\dots,d\}^{\bb N})$ such that
\begin{equation*}
(P_{T}^V)^*(\nu_V)= e^{{\lambda_V}  \, T} \nu_V\,,\quad\forall T\geq 0\,.
\end{equation*}
\end{theorem}

The proof of this theorem we will present in the Subsections \ref{eigenprobability} and \ref{eigenfunction}.

As a consequence of this theorem we will be able to normalize the  semigroup $\{P^V_T,\,\,T\geq 0\}$
 in order to get another stochastic semigroup, and, then we will finally  obtain what we call the Gibbs state in the continuous time setting.

A quite simple version of this result was presented in \cite{BEL}.
In this paper, $V$ depends just on $X_0$ and the state space  is $\{1,2,\dots,d\}.$

\begin{example} To clarify ideas, we present a simple example where is easy to verify the validity of the above theorem.
 Given $0<p_1<1$,  $0<p_2<1$, the stochastic matrix
$$
\left(
\begin{array}{cc}
1 - p_1  & p_1\\
p_2 & 1- p_2
\end{array}
\right)\,,$$
defines a Ruelle operator $\mathcal{L}_A$ acting on the one-dimensional spin lattice $\{1,2\}^{\bb N}$ such that, $\mathcal{L}_A(1)=1$.
More precisely, $e^{ A(1,1,x_2,\dots)} = 1 - p_1$,  $e^{ A(2,1,x_2,\dots)} = p_1$ and
 $e^{ A(1,2,x_2,\dots)} = p_2$,  $e^{ A(2,2,x_2,\dots)} = 1-p_2$.
Notice that $$
L= \left(
\begin{array}{cc}
1 - p_1  & p_1\\
p_2 & 1- p_2
\end{array}
\right)- \left(
\begin{array}{cc}
1   & 0\\
0 & 1
\end{array}
\right) =\left(
\begin{array}{cc}
- p_1  & p_1\\
p_2 & - p_2
\end{array}
\right)$$
defines a line sum zero matrix.
One can consider a potential $V$ such that is constant in the cylinders of size one, i.e.,
$V(1,x_1,x_2,\dots)=V_1$, and $V(2,x_1,x_2,\dots)=V_2$
In this case $L+ V$ is the matrix
$$\left(
\begin{array}{cc}
- p_1 + V_1 & p_1\\
p_2 & - p_2+V_2
\end{array}
\right)$$
If $\,V_1,V_2$ are positive and large then the positive cone goes inside the positive cone.
Then, there is a positive eigenvalue and a positive eigenfunction.
One can add a constant to $V$ in order to get an eigenvector with just positive entries.
\end{example}

We will consider on the Bernoulli space the usual metric $d$.  Let $0<\theta<1$, then for all $x=\{x_i\},y=\{y_i\}\in\{1,\dots,d\}^{\bb N}$
\begin{equation*}
 d(x,y):= \theta^N,
\end{equation*}
where $N$ is such that $x_i=y_i$, $\forall i\leq N$ and  $x_{N+1}\neq y_{N+1}$.
In the following, when $a\in\{1,\dots,d\}$ and  $x\in\{1,\dots,d\}^{\bb N}$ the notation $ax$
 means $(a,x_1,x_2,\dots)\in\{1,\dots,d\}^{\bb N}$, i.e.,
$ax$ is a preimage of $x$ by shift operator.

We point out that $ d(ax,ay)\leq \theta d(x,y)$, for all $x,y\in\{1,\dots,d\}^{\bb N}$ and $a\in\{1,\dots,d\}$, this is a
central idea in Lemma \ref{lemma1}, when we estimate the ratio $\frac{P_{T}^V (f)(x)}{P_{T}^V (f)(y)}$.
First, we will characterize
the operator $P_{T}^V$, in  Lemma \ref{pv}. This characterization allow us to conclude that the family of operators $\{P^V_T,\,\,T\geq 0\}$ describes
a natural generalization of the discrete time Ruelle operator (see \cite{BEL}).

\begin{lemma}\label{pv} Let $f\in\mc C$, $T\geq 0$, and $x\in \{1,\dots,d\}^{\bb N}$.
Consequently, $P_{T}^V (f)(x)\,=\, \bb E_{x} \big[e^{\int_0^{T} V(X_r)\,dr} f(X_{T})\big]$ can be rewritten as
 \begin{equation*}
\begin{split}
e^{T V(x)}
 f(x)e^{-T}\,+\,
\sum_{n=1}^{+\infty} \sum_{a_1=1}^{d}\dots \sum_{a_n=1}^{d}
e^{A(a_1x)}\dots e^{A(a_n\dots a_1x)}
f(a_n\dots a_1x)\,  \mc I_V^T(a_n\dots a_1 x)\,,\\
\end{split}
\end{equation*}
where
 \begin{equation*}
\begin{split}
\mc I_V^T(a_n\dots a_1 x)\!=\!\int_0^\infty \!\!\!\!\!dt_{n}\!\dots \!\int_0^\infty\!\! \!\!\!dt_{0}\,
e^{t_0 V(x)+\dots+(T-\sum_{i=0}^{n-1}t_i) V(a_n\dots a_1x)}
\textbf{1}_{[\sum_{i=0}^{n-1}t_i\leq T< \sum_{i=0}^{n}t_i]}e^{-t_0}\dots e^{-t_n}\,.
\end{split}
\end{equation*}
\end{lemma}

As the proof of this lemma is very technical
we present it in Appendix \ref{provalema}.

Observe that, if one consider $V\equiv 0$, the previous lemma says that
\begin{equation*}
\begin{split}
 &P_{T}(f)(x)\,=\, \bb E_{x} \big[ f(X_{T})\big]\\&=
 e^{-T}\Big\{f(x)\,+\,
\sum_{n=1}^{+\infty} \frac{T^n}{n!}\sum_{a_1=1}^{d}\dots \sum_{a_n=1}^{d}
e^{A(a_1x)}\dots\, e^{A(a_n\dots a_1x)}
f(a_n\dots a_1x)\Big\}\\&=
 e^{-T}\Big\{f(x)\,+\,
\sum_{n=1}^{+\infty} \frac{T^n}{n!}(\mc L^n_A(f))(x)\Big\}\,,
\end{split}
\end{equation*}
because
\begin{equation*}
 \int_0^\infty \!\!\!\!dt_{n} \dots \int_0^\infty\!\! \!\!dt_{0}\,\,
\textbf{1}_{[\sum_{i=0}^{n-1}t_i\leq T< \sum_{i=0}^{n}t_i]}e^{-t_0}\dots e^{-t_n}\,=\,e^{-T}\,\frac{T^n}{n!}\,.
\end{equation*}

Thus,
$P_{T}(f)(x)\,= \frac{1}{e^T} e^{T \mathcal {L}_A}(f)(x)$, which is in accordance with the fact that
$\{P_T,\, T\geq 0\}$ is the semigroup associated to the generator $L=\mc L_A-I$.

\begin{lemma}\label{auxlema}
For any non-negative continuous function $f$ such that there exist $x\in \{1,\dots,d\}^{\bb N}$ and $T>0$ satisfying $P_{T}(f)(x)=0$,
we have that $f\equiv 0$.
\end{lemma}
\begin{proof}
 By the Lemma \ref{pv}, $f(a_n,\dots,a_1x)=0$, for all $a_i\in\{1,2,\dots, d\}$, $i=1,\dots,n$, for any $n\in \bb N$.
 Then $f(z)=0$, for any $z\in \{y;$ there exists $n$ such that $\sigma^n(y)=x\}$.
 But this set is dense in $\{1,\dots,d\}^{\bb N}$ and $f$ is continuous, thus $f(z)=0$, for any $z\in \{1,\dots,d\}^{\bb N}$.
\end{proof}

\begin{lemma}\label{lemma1} If the function $f$ satisfies
$f(x)\leq e^{C_f d(x,y)} f(y)$, for all  $x,y\in\{1,\dots,d\}^{\bb N}$, where $C_f$ is a constant depending only on $f$,
 then
 \begin{equation*}
\begin{split}
P_{T}^V (f)(x)
\,\leq\, \exp\Big\{\big[(C_A\theta+TC_V)(1-\theta)^{-1}+C_f\theta\big]\,\mbox{d}(x,y)\Big\}\,P_{T}^V (f)(y)\,,
\end{split}
\end{equation*}
for all $T\geq 0$.
\end{lemma}

The proof of this lemma is in Appendix \ref{provalema} (it  is  similar to the proof of the Lemma \ref{pv}).

\subsection{Eigenprobability}\label{eigenprobability}

In this subsection we will present the proof of existence of eigenprobability.
Without loss of generality, we will assume that the perturbation $V$ is  positive and its minimum is large enough (just add a large constant to the initial $V$).
We will find an eigenprobability for $\mathcal{L}_A - I  + V$. The constant we eventually add to the in initial potential will not harm our argument.

First we need to analyze the dual of $\mathcal{L}_A - I  + V$ acting on signed measures.

As we know $ (\mathcal{L}_A - I  + V)^*$ acts on measures on the Bernoulli space via the expression: given $\nu$, then
\begin{equation*}
\big<f,(\mathcal{L}_A - I  + V)^*(\nu)\big>= \big<(\mathcal{L}_A - I  + V)(f),\nu\big>\,,
\end{equation*}
for any $f\in\mc C$. This leads us to consider the operator $G$  on probabilities of the one-dimensional spin lattice.
Given $\nu$ probability on $\{1,\dots,d\}^{\bb N}$, $G$  acts on $\nu$ as
\begin{equation*}
\big<f,G(\nu)\big>=\frac{ \big<(\mathcal{L}_A - I  + V)(f),\nu\big>}{\big<(\mathcal{L}_A - I  + V)(1),\nu\big>}=\frac{ \big<(\mathcal{L}_A - I  + V)(f),\nu\big>}{\big<V,\nu\big>}\,,
\end{equation*}
for any $f\in\mc C$. The function $G$ is well defined by the hypothesis on $V$.
This $G$ is continuous, because it is the ratio of two continuous functions.
From  Schauder-Tychonoff Theorem, we get the existence of a fixed point probability $\nu_V$ for $G$.
Therefore, there exists $\lambda_V=\int V\, \mbox{d}\nu_V$ such that
\begin{equation*}
\int (\mathcal{L}_A - I  + V)(f)\,\mbox{d}\nu_V=\big<(\mathcal{L}_A - I  + V)(f),\nu_V\big> =\lambda_V\, \big<f,\nu_V\big>=
\lambda_V \int f\,\mbox{d} \nu_V\,,
\end{equation*}
for any $f\in\mc C$.
Since $L=\mc L_A -I$, we have
\begin{equation}\label{int0}
\int (L  + V-\lambda_V)(f)\,\mbox{d}\nu_V=0\,,
\end{equation}
for any $f\in\mc C$. By Feynman-Kac, the semigroup associated to operator $L  + V-\lambda_V$ is $\frac{P_T^V}{e^{\lambda_VT}}$.
Using the Trotter-Kato Theorem (see chapter IX section 12 in \cite{Y}), we get
$$\frac{P_T^V(f)}{e^{\lambda_VT}}=\lim_{n\to\infty}\Big(I-\frac{T}{n}(L+ V-\lambda_V)\Big)^n(f)\,.$$
Observe that is true
\begin{equation*}
\int \Big(I-\pfrac{T}{n}(L+V-\lambda_V)\Big)^n(f) \,\mbox{d}\nu_V = \int f \,\mbox{d}\nu_V\,,\quad\forall n,
\end{equation*}
and, this is a consequence  of two properties: the first one is that when the operator $L+ V-\lambda_V$
acts on $\mc C$ its image is contained $\mc C$ too; the second one is the equality \eqref{int0}. By Dominated Convergence Theorem, we get
\begin{equation}\label{automed}
\int \frac{P_T^V(f)}{e^{\lambda_VT}} \,\mbox{d}\nu_V\,=\, \int f \,\mbox{d}\nu_V\,,
\end{equation}
for any $f\in\mc C$. Consequently,
\begin{equation*}
\int f\,\mbox{d}\big[(P_T^V)^*(\nu_V)\big] \,= \,e^{\lambda_VT}\int f \,\mbox{d}\nu_V\,,
\end{equation*}
for any $f\in\mc C$.

\subsection{Eigenfunction}\label{eigenfunction}
Here, we present the existence of an eigenprobability.

Suppose that $\theta\leq 1/2$. Let
\begin{equation*}
 \Lambda=\big\{f\in \mc C;\,\,0\leq f\leq 1\,\,\mbox{and}\,\,f(x)\leq \exp\{\pfrac{C_A+C_V}{1-\theta} d(x,y)\} f(y),
 \,\forall x,y\in\{1,\dots,d\}^{\bb N}\big\}\,.
\end{equation*}
The set $ \Lambda$ is convex,
because for all $f,g\in\Lambda$ and $t\in(0,1)$
\begin{equation*}
 tf(x)+(1-t)g(x)\leq
\exp\{\pfrac{C_A+C_V}{1-\theta} d(x,y)\}\big(tf(y)+(1-t)g(y)\big)\,.
\end{equation*}
Let $\{f_n\}\subset\Lambda$, then $\Vert f_n\Vert_\infty\leq 1$ and
\begin{equation*}
\begin{split}
  \vert f_n(x)-f_n(y)\vert \,\leq\, &\Vert f_n\Vert_\infty \Big(  \exp\{\pfrac{C_A+C_V}{1-\theta} d(x,y)\} -1\Big)\\
\,\leq\, & \pfrac{C_A+C_V}{1-\theta} d(x,y) \exp\{\pfrac{C_A+C_V}{1-\theta} \}\,,
\end{split}
\end{equation*}
for all $n\in\bb N$. By Arzel\`a-Ascoli Theorem,  the sequence  $\{f_n\}$ has a limit point. Therefore, $\Lambda$ is a compact set.

\bigskip

 By the Lemma \ref{lemma1}, for all  $f\in \Lambda$, we have
 \begin{equation*}
\begin{split}
P_{T}^V (f)(x)
\,\leq\, \exp\Big\{\big[\pfrac{C_A\theta+TC_V}{1-\theta}+\pfrac{C_A+C_V}{1-\theta} \theta\big]\,\mbox{d}(x,y)\Big\}\,P_{T}^V (f)(y)\,,\quad \forall T\geq 0\,.
\end{split}
\end{equation*}
Take $T\leq \theta$, then
 \begin{equation*}
\begin{split}
P_{T}^V (f)(x)
\,\leq\, \exp\Big\{2\pfrac{C_A+C_V}{1-\theta} \theta\,\mbox{d}(x,y)\Big\}\,P_{T}^V (f)(y)
\,\leq\, \exp\Big\{\pfrac{C_A+C_V}{1-\theta} \,\mbox{d}(x,y)\Big\}\,P_{T}^V (f)(y)\,.
\end{split}
\end{equation*}
The last inequality is due to the assumption about $\theta$. Unfortunately, $P_{T}^V (f)$ can be greater than one, then we need to define
for all  $n\in \bb N$, the  operator $Q_T^n$ that acts on  $ g\in\Lambda$ as
\begin{equation*}
 Q_T^n(g):=\frac{P_{T}^V (g+1/n)}{\big\Vert P_{T}^V (g+1/n)\big\Vert_\infty}\,.
\end{equation*}
Notice that, for all $n\in \bb N$, the function  constant equal to $1/n$ belongs to $ \Lambda$, then
$$P_{T}^V (1/n)(x) \leq \exp\{\pfrac{C_A+C_V}{1-\theta} \,\mbox{d}(x,y)\}P_{T}^V (1/n)(y)\,,$$
for all $ T\in[0,\theta]$.
This allows us to show that $ Q_T^n:\Lambda\to\Lambda$, for all $n\in\bb N$.

Since $ \Lambda$ is convex and a compact set, we can apply the Schauder-Tychonoff Fixed Point Theorem to each $ Q_T^n:\Lambda\to\Lambda$
and see that there exists  $h_n^T\in \Lambda$ such that
\begin{equation}\label{hn}
\frac{P_{T}^V (h_n^T+1/n)}{\big\Vert P_{T}^V (h_n^T+1/n)\big\Vert_\infty}=h_n^T\,,\,\,\,\,\forall n\,,\,\,\,\,\forall  T\in[0,\theta]\,.
\end{equation}

Now, for fixed $ T\in[0,\theta]$, there exists $F_T\in \Lambda$ a limit point of the sequence  $\{h_n^T\}_n\subset\Lambda$, because $\Lambda$ is compact.
By the continuity of the operator $P_{T}^V $, the expression above becomes
\begin{equation}\label{autofuncao}
P_{T}^V (F_T)=\big\Vert P_{T}^V (F_T)\big\Vert_\infty\,F_T\,,\,\,\,\,\forall T\in[0,\theta]\,.
\end{equation}

First of all, we would like to prove that $F_T>0$. Hence, we begin to analyze the norm $\big\Vert P_{T}^V (F_T)\big\Vert_\infty$.
 By the equation \eqref{hn}, we have
\begin{equation*}
\big\Vert P_{T}^V (h_n^T+1/n)\big\Vert_\infty\,h_n^T(x)= \bb E_{x} \big[e^{\int_0^{T} V(X_r)\,dr} (h_n^T+1/n)(X_{T})\big]
\geq \Big [\big(\inf \, h_n^T\big)+1/n\Big]\,e^{-T\Vert V\Vert_\infty}\,,
\end{equation*}
for all $x$.
Then,
\begin{equation*}
\Big(\big\Vert P_{T}^V (h_n^T+1/n)\big\Vert_\infty-e^{-T\Vert V\Vert_\infty}\Big)\,
\inf \, h_n^T\,
\geq\, (1/n) \,e^{-T\Vert V\Vert_\infty}\,>\,0\,.
\end{equation*}
This implies that
\begin{equation*}
\big\Vert P_{T}^V (h_n^T+1/n)\big\Vert_\infty\,>\,e^{-T\Vert V\Vert_\infty}\,,\quad\forall n\,.
\end{equation*}
Recalling that $F_T$ is a limit point of $\{h_n^T\}_n$, the last inequality is transformed in
\begin{equation}\label{norma}
\big\Vert P_{T}^V (F_T)\big\Vert_\infty\,\geq\,e^{-T\Vert V\Vert_\infty}\,.
\end{equation}

Finally, suppose that $F_T(x_0)=0$, for some $x_0\in\{1,\dots,d\}^{\bb N}$. Due to the fact that $F_T$ is eigenfunction of the operator $P_T^V$, we have $P_T^V(F_T)(x_0)=0$.
Using the Lemma \ref{auxlema}, we get that $F_T\equiv 0$. But it is a contraction in relation to \eqref{norma}, because $P_T^V$ is linear.
As a result $F_T> 0$.

\bigskip

Now, we will characterize the eigenvalue, in order to do this we use the eigenprobability $\nu_V$.
The equations \eqref{autofuncao} and \eqref{automed} together imply that
\begin{equation*}
\big\Vert P_{T}^V (F_T)\big\Vert_\infty\int F_T \,\mbox{d}\nu_V\,=\,\int P_T^V(F_T)\,\mbox{d}\nu_V \,= \,e^{\lambda_VT}\int F_T \,\mbox{d}\nu_V\,.
\end{equation*}
Since $F_T\geq 0$, we get $\Vert P_{T}^V (F_T)\Vert_\infty=e^{\lambda_VT}$,
and using \eqref{autofuncao} one can conclude $ P_{T}^V (F_T)= e^{\lambda_VT}\,F_T$, $\forall  T\in[0,\theta]$.

\bigskip

The next step is to prove that $e^{\lambda_VT}$ is a simple eigenvalue for $ P_{T}^V$. We suppose that for each $  T\in[0,\theta]$
there exists $G_T$ such that
$ P_{T}^V (G_T)= e^{\lambda_VT}\,G_T$.
 Define
$\alpha_0^T:=\inf_{x}\frac{G_T(x)}{F_T(x)}$. Since the Bernoulli space is compact, there exist $x_0\in\{1,\dots,d\}^{\bb N}$ such that
$G_T(x_0)-\alpha_0^T F_T(x_0)=0$. Observe that $H_T:=G_T(x)-\alpha_0^T F_T(x)$ is a non-negative eigenfunction of $P_T^V$. Then
$P_T^V(H_T)(x_0)=0$. By Lemma \ref{auxlema}, $H_T\equiv 0$. Thus, $G_T$ is a scalar multiple of $F_T$. This shows that
$e^{\lambda_VT}$ is a simple eigenvalue.

\bigskip

We will try to eliminate the dependence on $T\in[0,\theta]$ in the functions $F_T$. Recall that $\theta\leq 1/2$.
Let $n_0:=\min\{n;\;2^{-n}\leq \theta\}$. Denote by $F:= F_{2^{-n_0}}$.
We claim that $ P_{2^{-n}}^V (F)= e^{\lambda_V2^{-n}}\,F$, $\forall n\geq n_0$.
To prove this note that  by the semigroup property we have that $P_{2^{-n_0}}^V (F_{2^{-n}})$ can be rewritten as $P_{2^{-n}}^V \dots P_{2^{-n}}^V(F_{2^{-n}})$, $\forall n\geq n_0$.
Applying $2^{n-n_0}$ times the fact that $F_{2^{-n}}$ is eigenfunction of the operator $P_{2^{-n}}^V$, we have
$ P_{2^{-n_0}}^V (F_{2^{-n}})
\,=\,e^{\lambda_V2^{-n_0}}\,F_{2^{-n}}$, $\forall n\geq n_0$.
 Since $e^{\lambda_V2^{-n_0}}$ is simple eigenvalue to the operator $P_{2^{-n_0}}^V$, we get $F_{2^{-n}}=F$, $\forall n\geq n_0$.
This finishes the claim.

\bigskip

The last claim and  the fact that the semigroup $\{P_T^V,\,T\geq 0 \}$ is associated to the operator $L+V$  imply that
\begin{equation*}
 (L+V)(F)=\lim_{n\to\infty}\frac{P_{2^{-n}}^V(F)-F}{2^{-n}}=\lim_{n\to\infty}\frac{e^{\lambda_V2^{-n}}-1}{2^{-n}}F=\lambda_VF\,.
\end{equation*}
Since the operator $L+V=\mc L_A-I+V$ is a bounded operator, using the equality above we get
\begin{equation*}
 P_T^V(F)(x)=e^{T(L+V)}(F)(x)=\sum_{n=0}^\infty \frac{T^n}{n!}(L+V)^n(F)(x)=
\sum_{n=0}^\infty \frac{T^n}{n!}\lambda_V^nF(x)=e^{\lambda_V T}F(x)\,,
\end{equation*}
for any $T\geq 0$.

\bigskip

Therefore, with these final considerations, we finished the proof of one of our main results, which is  Theorem \ref{prop2} (Perron-Frobenius).
Notice that  $\lambda_V$ is both eigenvalue  for the eigenfunction
(see Section \ref{eigenfunction}) and also  eigenvalue for the dual operator (section \ref{eigenprobability}).

In terms of discrete time dynamics we just showed the following result:

\begin{corollary}\label{V}  Given a normalized Lipschitz potential $A: \{1,\dots,d\}^{\bb N}\to \mathbb{R}$  and a Lipschitz function
$V: \{1,\dots,d\}^{\bb N}\to \mathbb{R}$, there exists Lipschitz function $F =F_V:\{1,\dots,d\}^{\bb N}\to \mathbb{R}$  and $\lambda=\lambda_V$ such that, for all $x\in \{1,\dots,d\}^{\bb N}$
\begin{equation}\label{equation}
 \frac{\mc L_A (F)(x)}{F(x)}=\sum_{\sigma(y)=x} \, \frac{e^{A(y)}\, F(y)}{F(x)} = 1- V(x)+\lambda\,.
\end{equation}
\end{corollary}
Notice that the addition of  a constant to $V$   produces an additive change in  the eventual eigenvalue $\lambda$.

\section{ The continuous time Gibbs state for $V$}

From the Perron-Frobenius Theorem associated to $V$, we can define a new continuous time Markov chain which will be the Gibbs state for $V$.
Remember that $L+V={\mc L}_A -I+V$ generates the semigroup $\{P_{T}^V, \,T\geq 0\}$.

For $T\geq 0$, if one defines
\begin{equation}\label{qv}
 \mc P^{V}_T (f) (x)\,=\,\bb E_x\Big[ e^{\int_0^{T}\, V(X_r) dr}\, \pfrac{F(X_T)}{e^{{\lambda_V}  T}\, F(x)}\, f( X_T)\Big]=\frac{P^{V}_{T} (Ff) (x)}{e^{{\lambda_V}  T}\, F(x)}\,,
\end{equation}
where $F$ and $\lambda_V$ are the eigenfunction and the eigenvalue, respectivelly.
Then $\mc P^{V}_T( 1)(x) =1$, $\forall x \in\{1,\dots,d\}^{\bb N}$.
This will define the stochastic semigroup we were looking for. From this we will get a new continuous time Markov chain which will help
to define the Gibbs state for $V$.

We point out that $\frac{\mathcal{L}_A (F)}{F}(y) = 1-V(y) + \lambda_{V}=\gamma_V(x)>c>0$, for some positive $c$.
We can say that because $F$ and $\mathcal{L}_A (F)$ are  continuous strictly positive functions and the state space is compact.

From the above, it is natural to consider a new normalized Lipschitz potential $B_V$ and a function $\gamma_{V}$ defined   by
\begin{equation}\label{defs}
 \begin{split}
&  B_V(y):=\,A(y)-\log{(1-V(\sigma(y))+\lambda_{V})}+\log{F(y)}-\log{F(\sigma(y))}\,,\quad\forall y\in \{1,\dots,d\}^{\bb N}\\
&\mbox{and}\qquad\gamma_{V}(x):=1- V(x)+\lambda_{V}\,,\quad\forall x\in \{1,\dots,d\}^{\bb N}\,,\\
 \end{split}
\end{equation}
where $V$, $F$ and $\lambda_{V}$ were introduced before.

\begin{proposition} \label{LV} If  $V$ is a   Lipschitz function  we define the operator  $L^{V}$ acting on bounded mensurable
 functions $f:\{1,\dots,d\}^{\bb N}\to \bb R$ as
\begin{equation}\label{gerVF1}
 L^{V}(f)(x)=\gamma_{V}(x)\, \sum_{\sigma(y)=x}e^{B_{V}(y)}\big[f(y)-f(x)\big]\,,
\end{equation}
where $B_{V}(y)$ and  $\gamma_{V}$ are defined in \eqref{defs}. Then, this operator, $L^V$
 is the infinitesimal generator associated to a semigroup $\{\mc P^{V}_T,\,T\geq 0\}$ defined in
\eqref{qv}.
\end{proposition}

\begin{proof}
We begin proving that the $\{\mc P^{V}_T,\,T\geq 0\}$ is a semigroup. Recalling its definition, we get
\begin{equation*}
 \begin{split}
  \mc P^{V}_{t} (\mc P^{V}_{s}(  f)) (x) = \frac{P^{V}_{t} ( F \mc P^{V}_{s} ( f)) (x)}{e^{{\lambda_V}  t}\, F(x)}\,,
 \end{split}
\end{equation*}
we need to analyze $P^{V}_{t} ( F \mc P^{V}_{s}(  f)) (x)$. In this way,
\begin{equation*}
 \begin{split}
 &P^{V}_{t} ( F \mc P^{V}_{s} (f)) (x)=\bb E_x\Big[e^{\int_0^{t}\, V(X_r) dr}\, F(X_t)\mc P^{V}_{s}  (f)(X_t)\Big]\\&
=\bb E_x\Big[e^{\int_0^{t}\, V(X_r) dr}\, \pfrac{F(X_t)}{e^{{\lambda_V}  s}\, F(X_t)}\,
 P^{V}_{s}(F f)(X_t)\Big]= \pfrac{1}{e^{{\lambda_V}  s}}P^{V}_{t+s}(F f)(x)\,.
 \end{split}
\end{equation*}
One can conclude that $\{\mc P^{V}_T,\,T\geq 0\}$ is a semigroup.

To prove that  the infinitesimal generator \eqref{gerVF1} is associated to this semigroup, we need to observe that
\begin{equation*}
 \begin{split}
\frac{  \mc P^{V}_{t} ( f) (x)-f(x)}{t} = \frac{1}{e^{{\lambda_V}  t} F(x)}\Bigg(\frac{P^{V}_{t} ( Ff) (x)-(Ff)(x)}{t} \Bigg)
+f(x)\Bigg(\frac{e^{-{\lambda_V}  t}-1}{t}\Bigg)\,.
 \end{split}
\end{equation*}
Taking the limit as $t$ goes to zero the expression above converges to
\begin{equation}\label{LV1}
 \begin{split}
  \frac{1}{F(x)}(L+V)(Ff)(x)-f(x)\lambda
&=-\lambda f(x)+V(x)f(x)+ \pfrac{1}{F(x)}\,L(Ff)(x)\,,
 \end{split}
\end{equation}
which we denote by $L^V(f)(x)$.
Using the hypotheses about $V$ and equation \eqref{equation} of the Lemma \ref{V}, we get that $L^V(f)(x)$ is equal to
\begin{equation*}
 \begin{split}
\sum_{\sigma(y)=x}\pfrac{e^{A(y)}F(y)}{F(x)} f(y) - (1-V(x)+\lambda)f(x)=\sum_{\sigma(y)=x}\pfrac{e^{A(y)}F(y)}{F(x)} \big[f(y) -f(x)\big]\,.
 \end{split}
\end{equation*}
Again, we use the Lemma \ref{V} to obtain $\gamma_{V}(x)F(x)=\mc L_A(F)(x)$. Thus, the expression above can be rewritten as
\begin{equation*}
 \begin{split}
\gamma_{V}(x)\sum_{\sigma(y)=x}\pfrac{e^{A(y)}F(y)}{\mc L_A(F)(x)} \big[f(y) -f(x)\big]
&=\gamma_{V}(x)\sum_{\sigma(y)=x}e^{B_{V}(y)}  \big[f(y) -f(x)\big]
\,.
 \end{split}
\end{equation*}

\end{proof}

\begin{corollary}\label{corol9}
 For all $f\in \mc C^+$, $x\in\{1,\dots,d\}^{\bb N}$ and $t>0$ small
\begin{equation*}
 \begin{split}
\log \Big(\frac{  \mc P^{V}_{t} ( f)(x)}{f(x)}\Big)\,\sim\,\frac{tL^V(f)(x)}{f(x)}\,,
 \end{split}
\end{equation*}
where $a_n\sim b_n$ means that $a_n/b_n\to 1$, as $n\to \infty$.
\end{corollary}
\begin{proof}
 In the proof above we obtained that
\begin{equation*}
 \begin{split}
\lim_{t\to 0}\frac{  \mc P^{V}_{t} ( f) (x)-f(x)}{t}=L^V(f)(x)\,,
 \end{split}
\end{equation*}
where
$L^V(f)(x)=-\lambda f(x)+V(x)f(x)+ \pfrac{1}{F(x)}L(Ff)(x)$.
Then, for $t$ small
\begin{equation*}
 \begin{split}
\frac{  \mc P^{V}_{t} ( f)(x)}{f(x)}-1\,\sim\,\frac{tL^V(f)(x)}{f(x)}\,,
 \end{split}
\end{equation*}
for all $f\in \mc C^+$.
Since for all $x$ fixed and $t$ small we get
\begin{equation*}
 \begin{split}
\log \Big(\frac{  \mc P^{V}_{t} ( f)(x)}{f(x)}\Big)\,\sim\,\frac{  \mc P^{V}_{t} ( f)(x)}{f(x)}-1\,,
 \end{split}
\end{equation*}
we finished the proof.

\end{proof}

We will elaborate now on the  initial stationary probability $\mu_{B_{V},\gamma_{V}}$.
Notice that all of the above depends on the choice of the initial a priori
probability (which, in our case, is associated to the generator $L=\mathcal{L}_A-I$).
 The stationary measure for the continuous time process generated by $L^V$ (with exponential time of jump equal to
 $\gamma(x)=\gamma_{V}(x)=1- V(x)+\lambda_{V}$) is
\begin{equation}\label{muVgamma}
 \mbox{d}\mu_{B_{V},\gamma_{V}}(x)\, =\,\frac{1}{\gamma_{V}(x)}\,
\frac{\mbox{d}\mu_{B_{V}}(x)}{\int \frac{1}{\gamma_{V}}\, \mbox{d} \mu_{B_{V}}}\,,
\end{equation}
where $\mu_{B_{V}}$ is discrete time equilibrium  for the normalized Lipschitz potential
$B_{V}(y)=  \,A(y)+\log{F(y)}-\log{F(\sigma(y))}- \log \gamma_V(\sigma(y))$. In other words,
for any $f\in\mc C$, we have
\begin{equation*}
\int L^{V}(f) \, \,\mbox{d}\mu_{B_{V},\gamma_{V}}\,=\,0\,.
\end{equation*}

As we said before, the appearance of the term $\frac{1}{\gamma_V}$ introduce a new element,
which was not present in the classical discrete time setting.
\bigskip

\begin{definition} Given a Lipschitz function $V$,
we define a continuous time Markov process $\{Y^{V}_T, T\geq 0\}$ with state space $\{1,\dots,d\}^{\bb N}$ whose
infinitesimal generator $L^{V}$ acts on bounded mensurable functions $f:\{1,\dots,d\}^{\bb N}\to \bb R$ by the expression
\begin{equation}\label{gerVF}
 L^{V}(f)(x)=\gamma_{V}(x)\, \sum_{\sigma(y)=x}e^{B_{V}(y)}\big[f(y)-f(x)\big]\,,
\end{equation}
where $B_{V}$ and  $\gamma_{V}$ are defined in \eqref{defs}.
Now, we consider the  initial stationary probability  $\mu_{B_{V},\gamma_{V}}$ defined in \eqref{muVgamma}.
We call this process $\{Y^{V}_T, T\geq 0\}$ {\bf{\emph{ the continuous time Gibbs state for the potential $V$}}}.
This defines a probability $\bb P^V:=\bb P^V_{\mu_{B_{V},\gamma_{V}}}$ on the Skorohod space  $\mc D$ which we call {\bf {\emph{the Gibbs probability
 for the interaction $V$}}}.
\end{definition}

Notice that for $\{Y^{V}_T, T\geq 0\}$, the exponential time of jumping tends to be larger when we are close to the maximum of $V$.
For a generic continuous time path, the particle stays more time on this region.

If $V$ is of the form $-\frac{L(u)}{u}$, for some $u\in\mc C^+$, then, $\lambda=0$, and $\mu_A= \mu_{B_{V}}.$
In this case $\gamma=\frac{\mathcal{L}_A(u)}{u}.$

\section{Relative Entropy,  Pressure and the equilibrium state for $V$}

One can ask:   \textquotedblleft Did the Gibbs state of the last section satisfy a variational principle?"  We will address this question  in the present section.

\begin{definition}\label{admissible}
The probability $\tilde{\bb P}_\mu= \tilde{\bb P}_\mu^{\tilde{\gamma},\tilde{A} } $ on  $\mc D$
is called admissible, if it is generated by  the initial measure $\mu$ and the
continuous time Markov chain with infinitesimal generator $\tilde{L}$, which acts
on bounded mensurable functions $f:\{1,\dots,d\}^{\bb N}\to \bb R$ by
\begin{equation}\label{Ltilde}
\tilde{L} (f)(x)=\tilde{\gamma}(x)\, \sum_{\sigma(y)=x}\,e^{\tilde{A} (y)} \big[f(y)-f(x)\big]\,,
\end{equation}
where $\tilde{\gamma}$ is a strictly positive  continuous function, and, $\tilde{A}$ is a normalized Lipschitz potential.
We point out that $\mu$ do not have to be stationary for this chain.
\end{definition}

Notice that according to the last section all the Gibbs Markov chains $\bb P^V_{\mu_{B_{V},\gamma_{V}} }$
 one gets from a generic $V$  are admissible.
If we take any $\mu$ on $\{1,\dots,d\}^{\bb N}$, and  we denote by $\bb P_{\mu}$  the one we get when $\tilde{A}=A$ and $\tilde\gamma=1$,
 i.e.,
 the one we get from the unperturbed system with the initial measure $\mu$, then  $\bb P_{\mu}$ is also admissible.

In the same way as in \eqref{muVgamma}, the stationary measure
for the continuous time process with generator \eqref{Ltilde} is
\begin{equation}\label{mutilde}
 \mbox{d}\mu_{\tilde{A},\tilde{\gamma}}(x) =\frac{1}{\tilde{\gamma}(x)}\,
\frac{\mbox{d}\mu_{\tilde{A}}(x)}{\int \frac{1}{\tilde{\gamma}}\, \mbox{d} \mu_{\tilde{A}}},
\end{equation}
where $\mu_{\tilde{A}}$ is discrete time equilibrium  for
$\tilde{A}$.

From now on, we will consider a certain Lipschitz potential $V$ fixed until the end of this section.
The different probabilities $\tilde{\bb P}_{ \mu_{\tilde{A},\tilde{\gamma}}}^{\tilde{\gamma},\tilde{A} }$
 on $\mc D$ will describe the possible candidates for being the {\it stationary equilibrium continuous
time Markov chain for $V$} as we will explain later in our reasoning.

Given $V$ we
will consider  a {\it variational problem in the continuous time setting}
 which is analogous to the pressure problem in the discrete time setting (thermodynamic formalism).
This requires a meaning for {\it entropy}. A continuous time stationary Markov chain,
 which maximizes our variational problem, will be the {\it continuous time equilibrium state for $V$}.
By changing  $\tilde{\gamma}$  and $\tilde{A}$, we get a set of different infinitesimal
generators that are candidates to define {\it the continuous time equilibrium state} for
 the given potential $V$.
Nevertheless, it just makes sense to look for candidates among the admissible ones.
We will show in the end that the continuous time  equilibrium state for $V$
is indeed the Gibbs state $\bb P^V_{\mu_{B_{V},\gamma_{V}} }$ of the last section.

We will fix a certain $\mu$ on $\mc P(\{1,\dots,d\}^{\bb N})$ (no restrictions about it).
First, we want to give  a meaning for the  relative entropy of  any admissible probability $\tilde{\bb P}_\mu$  concerning $\bb P_\mu$.
The reason why we use the same  initial measure $\mu$ for both processes is that
we need  that the associated probabilities, $\tilde{\bb P}_\mu$  and $\bb P_\mu$,
on  $\mc D$ are  absolutely continuous with respect to each other.
Anyway, the final numerical result for the value of entropy  will not depend on the common $\mu$ we chose as the initial probability,
as can be seen in Lemma \ref{lema13}.
The common $\mu$ could de eventually $\mu_A$.
For a fixed $T\geq 0$, we consider the relative entropy
of the $\tilde{\bb P}_\mu=  \tilde{\bb P}_\mu^{\tilde{\gamma},\tilde{A} }$,
for some $\tilde{\gamma},\tilde{A}$, concerning $\bb P_\mu$ up to  time $T\geq 0$
by
\begin{equation}\label{entropy}
 H_T(\tilde{\bb P}_\mu\vert\bb P_\mu)\,=-\,\int_{\mc D}
\log\Bigg(\frac{\mbox{d}\tilde{\bb P}_\mu}{\mbox{d}\bb P_\mu}\Big|_{\mc F_T}\Bigg)(\omega)\,
\mbox{d}\tilde{\bb P}_\mu(\omega)\,.
\end{equation}

Using the property that the  logarithm is a concave function and Jensen's inequality, we obtain that
for any $g$ we have $\int \log g \,\mbox{d} \mu\leq \log \int g \,\mbox{d} \mu$. Then
$H_T(\tilde{\bb P}_\mu\vert\bb P_\mu)\leq 0$.
Negative entropies appear in a natural way when one analyzes a dynamical system with the property that
 each point has an uncountable number of preimages (see \cite{LMMS} and  \cite{LMST}).

By  Proposition \ref{apendice1} in Appendix \ref{RN}, the logarithm of the Radon-Nikodym derivative described above can be written as
\begin{equation*}
\log\Bigg(\frac{\mbox{d}\tilde{\bb P}_\mu}{\mbox{d}\bb P_\mu}\Big|_{\mc F_T}\Bigg)(\omega)
\end{equation*}
\begin{equation}\label{logRN}\,=\,
\int_0^T [1-\tilde{\gamma}(\omega_s)]\,\mbox{d}s+\sum_{s\leq T}\textbf{1}_{\{\sigma(\omega_s)=\omega_{s^-}\}}
\big[\tilde{A}(\omega_s)-A(\omega_s)+\log\big(\tilde{\gamma}(\sigma(\omega_s))\big)\big]\,.
\end{equation}

\begin{lemma}\label{lemma11}
For all $G\in\mc C$, it is true that
\begin{equation*}
\begin{split}
 \int_{\mc D}\sum_{s\leq T}\textbf{1}_{\{\sigma(\omega_s)\,=  \,\omega_{s^-}\}} G(\omega_s)
\,\mbox{d}\tilde{\bb P}_\mu(\omega)
= & \,\int_{\mc D}\int_0^T\tilde \gamma(\omega_s)G(\omega_s)\, \mbox{d}s\,\,
\mbox{d}\tilde{\bb P}_\mu(\omega)\, \\
=&\, \int_0^T\int_{\{1,\dots, d\}^{\bb N}}\tilde{P}_s(\tilde{\gamma} G)(x)\, \mbox{d}\mu(x)\, \mbox{d}s\,,\\
\end{split}
\end{equation*}
where $\{\tilde{P}_s,\,s\geq 0\}$ is the semigroup associated to the Markov chain that it was generated by $\tilde{L}$, see \eqref{Ltilde}.
\end{lemma}
The proof of this lemma is  in Appendix \ref{provalema11}.

Now, from \eqref{entropy}, \eqref{logRN} and the lemma above we obtain
\begin{equation}\label{17}
\begin{split}
  H_T(\tilde{\bb P}_\mu\vert\bb P_\mu)\,=\,&
\int_0^T\int_{\{1,\dots, d\}^{\bb N}}\tilde{P}_s(\tilde{\gamma}-1)(x)\, \mbox{d}\mu(x)\, \mbox{d}s\\&+
\int_0^T\int_{\{1,\dots, d\}^{\bb N}}\tilde{P}_s(\tilde{\gamma}[A-\tilde{A}-
\log\tilde{\gamma}\circ\sigma])(x)\, \mbox{d}\mu(x)\, \mbox{d}s\,.
\end{split}
\end{equation}
From the previous expression and ergodicity we get that there exists the limit $\lim_{T\to \infty} \frac{1}{T} H_T(\tilde{\bb P}_\mu\vert\bb P_\mu)$.

\begin{definition}\label{12}
For a fixed initial probability $\mu$ on $\mc P(\{1,\dots,d\}^{\bb N})$, the limit
\begin{equation*}
\lim_{T\to \infty} \frac{1}{T} H_T(\tilde{\bb P}_\mu\vert\bb P_\mu)
\end{equation*}
is  called the \emph{relative entropy} of the measure $\tilde{\bb P}_\mu$ concerning the measure $\bb P_\mu$
 (recall that $\bb P_\mu$ is associated to the initial fixed potential $A$).
Moreover, we denote this limit by $H(\tilde{\bb P}_\mu\vert\bb P_\mu)$.
\end{definition}
The goal of the next result is characterize the \emph{relative entropy} of the measure $\tilde{\bb P}_\mu$ concerning  $\bb P_\mu$.
\begin{lemma}\label{lema13}
The relative entropy $ H(\tilde{\bb P}_\mu\vert\bb P_\mu)$ can be written as
 \begin{equation*}
\begin{split}
\,&
\int_{\{1,\dots, d\}^{\bb N}}(\tilde{\gamma}(x)-1)\, \mbox{d}\mu_{\tilde{A},\tilde{\gamma}}(x) \\&\,\,\,+\,
\int_{\{1,\dots, d\}^{\bb N}}\tilde{\gamma}(x)\,\big[A(x)-
\tilde{A}(x)-\log(\tilde{\gamma}\circ\sigma)(x)\big]\, \mbox{d}\mu_{\tilde{A},\tilde{\gamma}}(x)\, .
\end{split}
\end{equation*}
\end{lemma}

\begin{proof}
 This proof follows by Definition \ref{12}, expression \eqref{17} and Ergodic Theorem.
\end{proof}

\begin{definition} For $A$ fixed, and a given Lipschitz potential $V$, we denote the Pressure (or, Free Energy) of $V$ as the value
\begin{equation*}
\textbf{P}(V):= \sup_{\at{ \tilde{\bb P}_\mu}{\text{admissible}}}
\,H(\tilde{\bb P}_\mu\vert\bb P_\mu)\, +\,
\int_{\{1,\dots, d\}^{\bb N}}V(x)\, \mbox{d}\mu_{\tilde{A},\tilde{\gamma}}(x)\,,
\end{equation*}
where $\mu_{\tilde{A},\tilde{\gamma}}$ is the initial stationary probability for the infinitesimal generator $\tilde{L}$, defined in \eqref{Ltilde}.
Moreover, any admissible element which maximizes $\textbf{P}(V)$ is called a continuous time  equilibrium state for $V$.
\end{definition}

Finally, we can state the main result of this section:
\begin{proposition}
 The pressure of the potential $V$ is given by
\begin{equation*}
\textbf{P}(V)\,=\,H(\bb P_\mu^{V}\vert\bb P_\mu)+
\int_{\{1,\dots, d\}^{\bb N}}V(x)\,\mbox{d}\mu_{B_{V},\gamma_{V}}(x)=\lambda_{V}\, .
\end{equation*}

Therefore, the equilibrium state for $V$ is the Gibbs state for $V$.
\end{proposition}

\begin{proof}

Recalling the definition of the measure $\mu_{\tilde{A},\tilde{\gamma}}$ in \eqref{mutilde} and the fact that the measure
$\mu_{\tilde{A}}$ is invariant for the shift, we get  that the second term
in \eqref{17} can be rewritten as
\begin{equation*}
\begin{split}
\Big[\int \pfrac{1}{\tilde{\gamma}}\, \mbox{d} \mu_{\tilde{A}}\Big]^{-1}\int_{\{1,\dots, d\}^{\bb N}}\big(A(x)-
\tilde{A}(x)\big) \mbox{d}\mu_{\tilde{A}}(x)\,
-\,\int_{\{1,\dots, d\}^{\bb N}}
\tilde{\gamma}(x)\,\log\tilde{\gamma}(x)\, \mbox{d}\mu_{\tilde{A},\tilde{\gamma}}(x)
\, .
\end{split}
\end{equation*}
Let $V$ be a Lipschtz function. Thus,
\begin{equation*}
\begin{split}
 &H(\tilde{\bb P}_\mu\vert\bb P_\mu)+\int_{\{1,\dots,d\}^{\bb N}}V(x)\,
 \mbox{d}\mu_{\tilde{A},\tilde{\gamma}}(x)\\&=\,
\int_{\{1,\dots, d\}^{\bb N}}\big(\tilde{\gamma}(x)-\tilde{\gamma}(x)\,\log\tilde{\gamma}(x)-1+V(x)\big)\,
 \mbox{d}\mu_{\tilde{A},\tilde{\gamma}}(x) \\
&\,\,\,\,\,\,\qquad\quad+\,
\Big[\int \pfrac{1}{\tilde{\gamma}}\, \mbox{d} \mu_{\tilde{A}}\Big]^{-1}\int_{\{1,\dots, d\}^{\bb N}}\big(\,A(x)-
\tilde{A}(x)\big) \mbox{d}\mu_{\tilde{A}}(x)
\, .
\end{split}
\end{equation*}
From equation \eqref{equation}, we can express the function $V$ as $ \lambda_{V}+1-\gamma_{V}(x)$.
Then the expression above becomes
\begin{equation}\label{entr+V}
\begin{split}
 &H(\tilde{\bb P}_\mu\vert\bb P_\mu)+\int_{\{1,\dots,d\}^{\bb N}}V(x)\,
 \mbox{d}\mu_{\tilde{A},\tilde{\gamma}}(x)\\&
=\,\lambda_{V}\,+\,\Big[\int \pfrac{1}{\tilde{\gamma}}\, \mbox{d} \mu_{\tilde{A}}\Big]^{-1}
\int_{\{1,\dots, d\}^{\bb N}}\Big(1-\log\tilde{\gamma}(x)-\frac{\gamma_{V}(x)}{\tilde{\gamma}(x)}\Big)\,
 \mbox{d}\mu_{\tilde{A}}(x) \\
&\,\,\,\,\,\,\qquad\quad+\,
\Big[\int \pfrac{1}{\tilde{\gamma}}\, \mbox{d} \mu_{\tilde{A}}\Big]^{-1}\int_{\{1,\dots, d\}^{\bb N}}
\big(A(x)-\tilde{A}(x)\big) \,\mbox{d}\mu_{\tilde{A}}(x)
\, .
\end{split}
\end{equation}
The last integral above is equal to $\int A\,\mbox{d}\mu_{\tilde{A}}+h(\mu_{\tilde{A}})$.
In order to analyze the second term in \eqref{entr+V}, we add and subtract $\log\gamma_{V}(x)$ in the integrand and we
use  $1+\log y-y\leq 0$, for all
 $y\in(0,\infty)$. Thus,
\begin{equation*}
1-\log\tilde{\gamma}(x)-\frac{\gamma_{V}(x)}{\tilde{\gamma}(x)}\leq  -\log\pfrac{\mc L_A(F)(x)}{F(x)}\,,
\end{equation*}
because $\gamma_{V}(x)=\pfrac{\mc L_A(F)(x)}{F(x)}$, for any $x\in\{1,\dots,d\}^{\bb N}$.
This implies that
\begin{equation*}
\begin{split}
 &H(\tilde{\bb P}_\mu\vert\bb P_\mu)+\int_{\{1,\dots,d\}^{\bb N}}V(x)\,
 \mbox{d}\mu_{\tilde{A},\tilde{\gamma}}(x)
\\&\leq\,\lambda_{V}\,+\,
\Big[\int \pfrac{1}{\tilde{\gamma}}\, \mbox{d} \mu_{\tilde{A}}\Big]^{-1}\Bigg[
-\int_{\{1,\dots, d\}^{\bb N}}
\log\pfrac{\mc L_A(F)(x)}{F(x)}\,
\mbox{d}\mu_{\tilde{A}}(x)+\int_{\{1,\dots,d\}^{\bb N}} A\,\mbox{d}\mu_{\tilde{A}}+h(\mu_{\tilde{A}})\Bigg]
\, .
\end{split}
\end{equation*}
By \cite{artur} (see Theorem 4) and \cite{Ki1}, we have
\begin{equation*}
 \int_{\{1,\dots,d\}^{\bb N}} A\,\mbox{d}\mu_{\tilde{A}}+h(\mu_{\tilde{A}})\,=\,\inf_{u\in\mc C^+} \int_{\{1,\dots, d\}^{\bb N}}
\log\pfrac{\mc L_A(u)(x)}{u(x)}\,
\mbox{d}\mu_{\tilde{A}}(x)\,.
\end{equation*}
 Since $F\in \mc C^+$ and
 $\int \pfrac{1}{\tilde{\gamma}}\, \mbox{d} \mu_{\tilde{A}}>0$, we obtain
\begin{equation*}
\begin{split}
 &H(\tilde{\bb P}_\mu\vert\bb P_\mu)+\int_{\{1,\dots,d\}^{\bb N}}V(x)\,
 \mbox{d}\mu_{\tilde{A},\tilde{\gamma}}(x)\,\leq\,\lambda_{V}\,.
\end{split}
\end{equation*}

One special case is
when the measure $\tilde{\bb P}_\mu$ is $\bb P_\mu^{V}$,
i.e.,

 $$\tilde{\gamma}(x)= \gamma_{V}(x)=1-V(x)+\lambda_{V}=\frac{\mc L_A(F)(x)}{F(x)},$$
and
$$\tilde{A}(x)=B_{V}(x)= A(x)+\log F(x)-\log \mc L_A(F)(\sigma(x)).$$
In this case, the expression \eqref{entr+V} becomes
\begin{equation*}
\begin{split}
 H(\bb P_\mu^V&\vert\bb P_\mu)+\int_{\{1,\dots,d\}^{\bb N}}V(x)\,
 \mbox{d}\mu_{B_{V},\gamma_{V}}(x)\\
 =\,&\lambda_{V}
+\,\Big[\int \pfrac{1}{\gamma_{V}}\, \mbox{d}\mu_{B_{V}}\Big]^{-1}
\int_{\{1,\dots, d\}^{\bb N}}\Big[-\log\Big(\pfrac{\mc L_A(F)(x)}{F(x)}\Big)\\
& -\log F(x)+\log \mc L_A(F)(\sigma(x))\Big]\, \mbox{d}\mu_{B_{V}}(x) \,.
\end{split}
\end{equation*}
Due to the fact that  $\mu_{B_{V}}$ is an invariant measure for the shift, we finally get
\begin{equation*}
\begin{split}
 &H(\bb P_\mu^V\vert\bb P_\mu)+\int_{\{1,\dots,d\}^{\bb N}}V(x)\,
 \mbox{d}\mu_{B_{V},\gamma_{V}}(x)\,=\,\lambda_{V}\,.
\end{split}
\end{equation*}
\end{proof}

\section{A large deviation principle for the empirical measure}

Nice general references on this topic are \cite{DS} and \cite{Ki1}. We point out that the process we consider {\bf is not} reversible
differently from \cite{DL}.

This section is divided on two subsections. The first one deals with the existence and the uniqueness of equilibrium states and the second one is about large deviation properties.

\subsection{Existence and uniqueness of equilibrium states}

As before, we considered a fixed normalized Lipschitz potential $A$ and the corresponding infinitesimal generator $L=\mathcal{L}_A-I$.
In this subsection we will assume that the perturbation $V$ is a  Lipschitz function.
As we mentioned before (see Subsection \ref{eigenprobability}), for the given potential $V$,
one can find an eigenprobability $\nu_V$.
This means that there exists $\lambda_V=\int V\, \mbox{d}\nu_V$ such that
\begin{equation*}
\int (\mathcal{L}_A - I  + V)(f)\, \mbox{d}\nu_V=\lambda_V \int f\, \mbox{d} \nu_V\,,
\end{equation*}
for any $f\in\mc C$.
As usual, we denote $\gamma_V(x) = 1-V(x)  +\lambda_V.$
Notice that $\int \gamma_V(x) \, \mbox{d} \nu_V(x) = \int (1-V(x)+\lambda_V)\, \mbox{d}\nu_V \,=1\,=\,\int \mathcal{L}_A (1)\, \mbox{d}\nu_V$.

Remember that
$L^{V}(f)(x)=\gamma_{V}(x)\, \sum_{\sigma(y)=x}e^{B_{V}(y)}\big[f(y)-f(x)\big]$,
where
$$B_{V}(y)= A(y)+\log F(y)-\log F(\sigma(y))- \log (1-V(\sigma(x))+\lambda_V),$$
is the infinitesimal generator associated to a semigroup $\{\mc P^{V}_T, \,T\geq 0\}$.

Moreover, for all  $u\in\mc C$ it is true that
$$\int \mc P^{V}_{t} ( u)\, \mbox{d} \mu_{B_{V},\gamma_{V}} \,=\int u \,\mbox{d} \mu_{B_{V},\gamma_{V}},$$
where
$\mbox{d}\mu_{B_{V},\gamma_{V}}(x) =\frac{1}{\gamma_{V}(x)}\,
\frac{\mbox{d}\mu_{B_{V}}(x)}{\int \frac{1}{\gamma_{V}}\, \mbox{d}\mu_{B_{V}}}$.

\begin{lemma}
Suppose $F=F_V>0$ is the main eigenfunction of the operator  $\mathcal{L}_A - I +V$ with eigenvalue $\lambda_V$,
then $\mbox{d}\tilde{\nu}_V(x):=\frac{1}{F(x)}\mbox{d}\mu_{B_{V},\gamma_{V}}(x)=\frac{1}{F(x)}\frac{1}{\gamma_{V}(x)}\,
\frac{\mbox{d}\mu_{B_{V}}(x)}{\int \frac{1}{\gamma_{V}}\, \mbox{d}\mu_{B_{V}}}$ satisfies, for all $g\in \mc C$,
\begin{equation*}
 \int (\mathcal{L}_A - I  + V)(g)\, \mbox{d}\tilde{\nu}_V\,=\,\lambda_V \int g\, \mbox{d} \tilde{\nu}_V\,.
\end{equation*}
Therefore, $\tilde{\nu}_V$ is an eigenprobability for $(\mathcal{L}_A - I  + V)^*$.
Moreover, if we know that the initial stationary probability for $\{\mc P^{V}_{t}=e^{t L^V}, \,t\geq 0\}$ is
unique,
then the eigenprobability is unique.
\end{lemma}
\begin{proof}
It is known that
 \begin{equation*}
 \int L^V(f)\, \mbox{d}\mu_{B_{V},\gamma_{V}}\,=0\,, \quad\forall\, f\in \mc C\,.
\end{equation*}
We can consider an equivalent expression for $L^V(f)$, which is  in \eqref{LV1},
then for any $f\in \mc C$,  we have
$$
\int \pfrac{1}{F}(L+V)(Ff)\, \mbox{d}\mu_{B_{V},\gamma_{V}}= \lambda \,\int f\, \mbox{d}\mu_{B_{V},\gamma_{V}}\,.$$
Denote by $\tilde{\nu}_V$ the measure $\frac{1}{F}\mu_{B_{V},\gamma_{V}}$.
Given a   $g\in\mc C$, take $f=g/F$, thus,
$$
\int (L+V)(g)\,\mbox{d} \tilde{\nu}_V= \lambda \,\int g\,\mbox{d}\tilde{\nu}_V\,.$$
This shows the first claim, that is,  $\tilde{\nu}_V$ is the eigenprobability.
Suppose  that the initial stationary probability, $\mu_{B_{V},\gamma_{V}}$, for $\{e^{\,t \, L^V}, \,t\geq 0\}$ is
unique and $\tilde{\nu}_V $ is the eigenprobability. By hypothesis $F=F_V$ is the unique main eigenfunction for $L+V$.
We can reverse the above argument for the measure  $F_V\, \mbox{d}\tilde{\nu}_V$. Notice that each step is an equivalence. Therefore, one can show that
$$
 \int L^V(f)\,F_V\,\mbox{d}\tilde{\nu}_V \,=0\,, \quad\forall f\in \mc C\,.
$$
From the uniqueness we assumed above, we get $\frac{d\mu_{B_{V},\gamma_{V}}}{d \tilde{\nu}_V }= F_V$.
The final conclusion is that if
the initial stationary probability for the continuous time Markov chain associated to $V$ satisfies $\mu_{B_{V},\gamma_{V}}= F_V\,\nu_V$, then $\tilde{\nu}_V $ is unique.
\end{proof}

\begin{lemma} \label{F>0}
If there exists a function $F\in \mc C^+$ such that $(L +V)F=\lambda_V F$, then the functional acting on $\mc P( \{1,\dots,d\}^{\bb N})$ given by
\begin{equation}\label{Ieh}
I(\nu) :=     - \inf_{u\in\mc C^+}\, \int \frac{L(u)}{u}\,\mbox{d} \nu \geq 0,
\end{equation}
satisfies
$$\lambda_V =\sup_{\nu\in\mc P(\{1,\dots,d\}^{\bb N})}\, \,\Big(\int V \mbox{d} \nu\, -\, I(\nu)\Big)\,.$$
The supremum value above is  achieved on the probability $ \mu_{B_{V},\gamma_{V}}$.
Moreover, if for any Lipschitz $V$ all the above is true, then,
using the Legendre Transform, we obtain
\begin{equation*}
 \begin{split}
  I(\nu)=\sup_{V\in\mc C}\, \Big(\int V\, \mbox{d} \nu\, -\,\lambda_V\Big)=
  \sup_{\at{V\in\mc C}{\text{and}\,\, V\,\,\text{is Lipschitz}} }\Big(\int V\, \mbox{d} \nu\, -\,\lambda_V\Big)\,,
 \end{split}
\end{equation*}
for all $\nu$ probability on $\{1,\dots,d\}^{\bb N}$
and $I(\nu)=\infty$ in any other case.
\end{lemma}

\begin{proof}
We follow the reasoning described in Section 4 of \cite{Ki1} adapted to the present case.
First, we show that
\begin{equation}\label{eq.18}
 \lambda_V \geq \sup_{\nu\in\mc P(\{1,\dots,d\}^{\bb N})}\, \,\Big(\int V \mbox{d} \nu\, -\, I(\nu)\Big)\,.
\end{equation}
Let $\nu\in\mc P(\{1,\dots,d\}^{\bb N})$, by definition of the functional $I$, we get
\begin{equation*}
 \int V \,\mbox{d} \nu\, -\, I(\nu)\leq \int V\, \mbox{d} \nu\,+\, \int \frac{L(u)}{u}\,\mbox{d} \nu\,,\quad\forall u\in \mc C^+\,.
\end{equation*}
We will take $u=F$, where $F$ is the eigenfunction of the $P_T^V$,  then we obtain
\begin{equation*}
 \int V \,\mbox{d} \nu\, -\, I(\nu)\leq \int V\, \mbox{d} \nu\,+\, \int \frac{L(F)}{F}\,\mbox{d} \nu\,.
\end{equation*}
Using the equation \eqref{equation}, we can rewrite $\frac{L (F)}{F}$ as $-V+\lambda_V$,
then the inequality follows.

Now, we will show that
\begin{equation}\label{eq.19}
 \lambda_V \leq \sup_{\nu\in\mc P(\{1,\dots,d\}^{\bb N})}\, \,\Big(\int V \mbox{d} \nu\, -\, I(\nu)\Big)\,.
\end{equation}
Actually, we will prove that
$$\lambda_V \leq
\int V \,\mbox{d}\mu_{B_{V},\gamma_{V}} - I(\mu_{B_{V},\gamma_{V}})\,,$$
and this implies the inequality \eqref{eq.19}.

In order to show the above, we consider a general $u\in\mc C^+$.
Recalling the  expression of $L^V$, which is  in Proposition \ref{LV}, we get
\begin{equation*}
 \frac{L^V(u/F)}{u/F}=\frac{L(u)}{u}+V - \lambda_V\,.
\end{equation*}
As the infinitesimal generator of $\mc P^{V}_{t} $ is $L^V$, from  Corollary \ref{corol9}, we have
for all $u\in \mc C^+$ and $t$ small
\begin{equation*}
 \begin{split}
\frac{L^V(u/F)}{u/F}\,\sim\,\frac{1}{t}\log \Big(\frac{  \mc P^{V}_{t} ( u/F)}{u/F}\Big)\,.
 \end{split}
\end{equation*}
Using these two last expressions we get for any $u\in\mc C^+$
$$ \int\Big[\frac{L (u) }{u}+ V - \lambda_V \Big]\,\mbox{d}
\mu_{B_{V},\gamma_{V}}\sim \frac{1}{t}\,\int \log \Big( \frac{\mc P^{V}_{t} (u/F)}{ u/F }\Big)\,\mbox{d} \mu_{B_{V},\gamma_{V}}\,.$$
By Jensen's inequality for any $u\in\mc C^+$ and small $t>0$, we have the right-hand side in the last expression is   bounded from below by
$$ \frac{1}{t} \int \Big[\mc P^{V}_{t} \Big(\log \big(u/F\big)\Big) - \log  \big(u/F \big)\Big]\,\mbox{d} \mu_{B_{V},\gamma_{V}}=0\,.$$
The last equality is due to fact that $\mu_{B_{V},\gamma_{V}}$ is the invariant measure.
Therefore, we take the infimum among all $u\in\mc C^+$ in the above expression, and we get
$$ \inf_{u\in \mc C^+}\, \int\Big[\frac{L (u) }{u}+ V \Big]\,\mbox{d}\mu_{B_{V},\gamma_{V}}\,\geq  \lambda_V \,.$$
Thus, we finish the proof of the inequality \eqref{eq.19}. Consequently, using \eqref{eq.18} and \eqref{eq.19} one can conclude
the statement of the lemma.
The last claim follows from a standard procedure via the classical Legendre transform.

\end{proof}

We point out that indeed is true that for any Lipchitz $V$ there exist $F$ and $\lambda$ as above. Therefore, the conclusion of last result is true in our case (for the corresponding $I$).

\smallskip

In the future we will need the property that for each Lipchitz $V$ the probability which attains the maximal value  $\sup_{\nu\in\mc P(\{1,\dots,d\}^{\bb N})}\, \,\Big(\int V \mbox{d} \nu\, -\, I(\nu)\Big)\,$ is unique. In this direction we consider first  the following lemma.

\begin{lemma} For a fixed Lipschitz $V$ ,
if $\rho$ realizes
$$\lambda_V =
\int V\, \mbox{d} \rho - I(\rho),$$
then
$ (\mc P^{V}_{t})^{*} (\rho)=\rho$,  for all $t\geq 0$.

\end{lemma}

\begin{proof}
By hypothesis, we get
$$ \inf_{u\in \mc C^+}\int\frac{(L +V- \lambda_V )(u) }{u}\,\mbox{d} \rho=0\,.$$
We have to show that for any $f\in\mc C$ it is true
$ \int L^V (f) \, d \rho\,=\,0.$
The inspiration for the main idea of  this proof comes from the reasoning of Sections 2 and 3 in \cite{DV1}
(a little bit different from the last paragraph of the proof of Proposition 3.1 in \cite{Ki2}).

Recalling the definition of the operator $L^V$ given in Proposition \ref{LV}, and using the Corollary \ref{V}, we obtain the next equality
\begin{equation}\label{equf} L^V (f)(x)= \frac{1}{F(x)} \, \sum_{\sigma(y)=x} \, e^{A(y)}\, F(y)\, [\, f(y)-f(x)]=
\frac{\mathcal{L}_A(F\, f) (x)}{F(x)}- \mathcal{L}_A(F) (x)\, \frac{f(x)}{F(x)}\,.\end{equation}

We point out that as $F\in\mc C^+$ is the main eigenfunction of $L+V)$ with eigenvalue $\lambda_V$, then, $F$  realizes the infimum
$$  \inf_{u\in \mc C^+}\int\frac{(L +V- \lambda_V )(u) }{u}\,\mbox{d} \rho=0\,.$$

Given any $f\in\mc C$, take $\epsilon>0$ such that $\epsilon< \frac{1}{c}$, where $\Vert f\Vert_\infty\leq c$.
For this choice of
 $\epsilon$,  observe that  $F\, (1+\epsilon f)\in\mc C^+$. Denoting
$$ G(\epsilon) \,:=\, \int \frac{(L+V-\lambda_V)\,(F\, (1+\epsilon f) ) }{F\, (1+\epsilon f)}\, d \rho\geq 0,$$
we note that $G(\epsilon)$ takes its minimal value at $\epsilon=0.$
Now, taking derivative of $G$ with respect to $\epsilon$ and applying to the value $\epsilon =0$, we get from \eqref{equf}

\begin{equation*}
\begin{split}
0=\,G'(0)\,=\,&\int \frac{(L+V-\lambda_V)\,(F\,f)\, \,\,F}{F^2}\, -  \frac{(F\, f)\, \,\,(L+V-\lambda_V)\,(F)}{F^2}\,  d \rho\\
=\,& \int \frac{(\,( \mathcal{ L}_A - I)+V-\lambda_V)\,(F\,f)\, }{F}\, -  \frac{ f\, (\,( \mathcal{ L}_A - I)+V-\lambda_V)\,(F)}{F}\,  d \rho\\
=\,& \int [\,\frac{ \mathcal{ L}_A (F\,f)\, }{F}-  \frac{f\,  \mathcal{ L}_A (F)\, }{F}\,]\, d \rho=\int L^V (f)\, d \rho\,.
\end{split}
\end{equation*}

\end{proof}

Uniqueness will follow from the next result.

\begin{proposition}  When  $V$ is  Lipschitz function, there is only one $\rho$ which is the
initial stationary probability for the stochastic semigroup $\{\mc P^{V}_{t},\,t\geq 0\}$ generated by
$L^V$.
\end{proposition}

\begin{proof}
Suppose $\rho$ is such that
for all $t\geq 0$, we have
$ (\mc P^{V}_{t})^{*} (\rho)=\rho.$
This means that for any $f\in\mc C$, we have
$$ \int \mc P^{V}_{t}(f )\, \mbox{d} \rho= \int f\, \mbox{d} \rho\,.$$
This implies that $\int L^V (f)\, \mbox{d} \rho=0$, for any $f\in\mc C$. Using the expression \eqref{LV1} for $L^V$, the last integral becomes
$$ \int\Big[ \pfrac{1}{F}( \mathcal{L}_A-I+V)(F\,f)-\lambda_V\,f \Big]\, \mbox{d} \rho\, =\,0\,,$$
 for any $f\in\mc C$,
which is equivalent to
\begin{equation*}
 \begin{split}
  &\int f\,(1-V+ \lambda_V)\, \mbox{d} \rho\,=\,\int \frac{\mathcal{L}_A (Ff)}{F} \, \mbox{d} \rho=
  \int \frac{\mathcal{L}_A (Ff)}{F}\, \frac{1}{1-V+ \lambda_V}   \, (1-V+ \lambda_V)\, \mbox{d} \rho\,.
 \end{split}
\end{equation*}
We point out that it is known that $1-V+ \lambda_V$ is strictly positive. Consider
$B_V(y)=A(y)- \,\log [  1- V(\sigma(y))+\lambda]+ \log F(y) - \log F(\sigma(y))$ and consider the following Ruelle operator
$$f\,\,\to\,\,\mathcal{L}_{B_V}(f)(x)\,\,=\,\, \sum_{\sigma(y)=x} \, e^{B_V(y)}\, f(y)\,,$$
which satisfies $\mathcal{L}_{B_V}(1)=1.$ From classical results in thermodynamic formalism there is a unique $\tilde{\mu}$ such that
$\mathcal{L}_{B_V}^*(\tilde{\mu})=\tilde{\mu}.$ We will show that $ \mbox{d} \tilde{\mu}= (1-V+ \lambda_V)\,\, \mbox{d} \rho\,.$
Indeed, $\mathcal{L}_{B_V}^*(\tilde{\mu})=\tilde{\mu}$ means that for any $f$, we have
\begin{equation*}
 \begin{split}
  &\int \, f\, \mbox{d}\tilde{\mu}=
\int\sum_{\sigma(y)=x} \, e^{A(y)- \,\log [  1- V(\sigma(y))+\lambda]+ \log F(y) - \log F(\sigma(y))}\, f(y)\, \mbox{d}\tilde{\mu}(x)
=\int \frac{\mathcal{L}_A (Ff)}{F}\, \frac{1}{1-V+ \lambda_V}   \, \mbox{d} \tilde{\mu}\, ,
 \end{split}
\end{equation*}
for all $f\in \mc C$.
Since $\tilde{\mu}$ is unique, we get that $\rho$ is unique.
\end{proof}
The next theorem follows easily from  the last two results.

\begin{theorem}\label{teo20} For a fixed Lipschitz function $V$,
there is a unique  $\rho$ which realizes
$$\lambda_V =
\int V\, \mbox{d} \rho - I(\rho)\,.$$
Moreover, $\rho= \mu_{B_{V},\gamma_{V}}$, which is the   initial stationary probability for $L^V$, and
the measure $\bb P^V_{\mu_{B_{V},\gamma_{V}}}$ is  invariant for the continuous time semiflow  $\{\Theta_t, \,t\geq 0\}$ on the Skorohod space.

\end{theorem}

We consider now some general statements that  will be necessary in the next section.

In the case that there exists the eigenfunction $F$, it is possible to show that
\begin{equation*}
 \lim_{T \to \infty}\, \pfrac{1}{T}\,\log\int_{\mc D} e^{\int_0^{T} V(\omega_r)\,dr} \, \mbox{d}\bb P_{x}(\omega) = \lambda_{V}\,,
\end{equation*}
for all $x\in \{1,\dots,d\}^{\bb N}$. Indeed,
$
\log\int_{\mc D} e^{\int_0^{T} V(\omega_r)\,dr} \, \mbox{d}\bb P_{x}(\omega) $ can be written as
$\lambda_{V} T+\log\Big(F(x)\frac{P_T^V(1)(x)}{P_T^V(F)(x)}\Big)$.
Since the eigenfunction $F$ is strictly positive on a compact set,
the second term in the last sum is bounded above and below by constants that depend only on $F$. This proves the desired limit.

In a similar way as above, we obtain
\begin{equation}\label{lim}
  \lim_{T \to \infty}\, \pfrac{1}{T}\,\log\int_{\mc D} e^{\int_0^{T} V(\omega_r)\,dr} \, \mbox{d}\bb P_{\mu_A} (\omega)\,=\lambda_{V}.
\end{equation}

Remember that
the value $\lambda_V$ was obtained from $V$ as the one such that
$ (\mathcal{L}_A - I  + V)^* \nu_{V}=\lambda_V \nu_{ V}$, with
$\lambda_V=\int V\,\mbox{d}\nu_{ V}$, see the Subsection  \ref{eigenprobability}.

We consider below a general continuous  potential $V$.

\begin{lemma} \label{Vi}
 For all continuous function $V:\{1,\dots,d\}^{\bb N}\to\bb R$, there exists the limit
 \begin{equation*}
  \lim_{T \to \infty}\, \pfrac{1}{T}\,\log\int\int_{\mc D} e^{\int_0^{T} V(\omega_r)\,dr} \, \mbox{d} \bb P_{x}(\omega)\,\mbox{d} \mu_A(x)
  \,=\,\lim_{T \to \infty}\, \pfrac{1}{T}\,\log\int P_{T}^V(1)(x) \,\mbox{d} \mu_A(x)\,.
 \end{equation*}
We will denote this limit by $ Q(V)$.
\end{lemma}
\begin{proof}
 Notice that, for all $x\in\{1,\dots,d\}^{\bb N}$ and $T,S\geq 0$, it is true that
 \begin{equation*}
 \begin{split}
   &P_{T+S}^V(1)(x)\,=\,P_{T}^V\big(P_{S}^V(1)\big)(x)\,
  =\,\bb E_x\Big[ e^{\int_0^{T} V(X_r)\,dr}P_{S}^V(1)(X_T)\Big]\\
  &\leq\,
  \int P_{S}^V(1)(x) \,\mbox{d} \mu_A(x)\;\; \bb E_x\Big[ e^{\int_0^{T} V(X_r)\,dr}\Big]
  \leq\,
  \int P_{S}^V(1)(x) \,\mbox{d} \mu_A(x) \;\;
  P_{T}^V(1)(x) \,,
 \end{split}
 \end{equation*}
 $\mu_A$-a.s. in $x$.
Then the limit in the statement of this lemma follows by subadditivity.
\end{proof}

The above result is related to questions raised in (4.3) in \cite{Ki1} and (4.2.21) in \cite{DS}.
\bigskip

From the above we get the next lemma.

\begin{lemma}\label{Q(V)}
For any Lipschitz function $V$, we have  $Q(V)= \lambda_V.$
\end{lemma}

We will show several properties of $Q(V)$ in Appendix \ref{apQV}.
More precisely, we show that in our setting the expressions (2.1) and (2.2) in \cite{Ki1} are true.

\subsection{Large deviations}
In this subsection we will apply to our setting the general results stated in \cite{Ki1}.
The purpose of this subsection is to show that the large deviation principle at the level two (see (1.1) and (1.2)
in \cite{Ki1}) is true for the{\it{\emph{ a priori }process}}.
We will have to show that the hypothesis of Theorem 2.1 in \cite{Ki1} is  true in our setting.
General references for large deviations are  \cite{DZ}, \cite{DS}, \cite{dH}, \cite{Ellis}, \cite{FK} and \cite{Leo}.

First, we will present the sequence of definitions and statements of \cite{Ki1} in the particular case of our setting.
Recalling that $\{X_t,\,t\geq 0\}$ denotes  the\emph{ a priori } continuous time stochastic process with infinitesimal generator $L=\mathcal{L}_A -I$
and initial probability $\mu_A$. We denote by $\bb P_{\mu_{A}}$ the probability on the Skorohod space $\mc D$ associated to such
stationary process.
We will begin with the  occupational time for  $\{X_t, t\geq 0\}$.
Define, for all $t\geq 0$, $\omega \in \mc D$ and for any Borel subset $\Gamma$ of the $\{1,\dots,d\}^{\bb N}$,
\begin{equation*}
 L_t^\omega(\Gamma)=\frac{1}{t}\int_0^t \textbf{1}_{\Gamma}(X_s(\omega))\,\mbox{d}s\,.
\end{equation*}
Observe that for $t$ and $\omega$ fixed we have  $L_t^\omega$ is a measure on $\{1,\dots,d\}^{\bb N}$
and it is called \textit{empirical measure}.
Moreover, if we consider the canonical version of the process $\{X_t, t\geq 0\}$, we can rewrite the expression above above as
\begin{equation*}
 \int_{\{1,\dots,d\}^{\bb N}} \textbf{1}_{\Gamma}(y)\,L_t^\omega(dy)=\frac{1}{t}\int_0^t \textbf{1}_{\Gamma}(\omega_s)\,\mbox{d}s\,.
\end{equation*}

Fixing $t\geq 0$ and $\omega\in\mc D$, using  the fact that
$ L_t^\omega$ is a measure on $\{1,\dots,d\}^{\bb N}$, moreover, using the expression above and usual arguments for approximating
bounded (or positive) functions, we have
\begin{equation}\label{eq.21}
 \begin{split}
\int_{\{1,\dots,d\}^{\bb N}} f(y)\,L_t^\omega(dy)\,=\,\frac{1}{t}\int_0^t f(\omega_s)\,\mbox{d}s\,,
\end{split}
\end{equation}
for all $f:\{1,\dots,d\}^{\bb N}\to \bb R$ bounded (or positive) mensurable function.

Finally, by the Ergodic Theorem, for any $f\in\mc C^+$, we have
$$\lim_{t \to \infty}   \int_{\{1,\dots,d\}^{\bb N}} f(y)\,L_t^\omega(dy)\,=
\, \int_{\{1,\dots,d\}^{\bb N}} f(y)\,\mu_A(dy)\,,\quad\bb  P_{\mu_{A}}-\mbox{almost surelly in }\,\, \omega\,,$$
in other words, $\lim_{t\to\infty}L_t^\omega=\mu_A$, $\bb  P_{\mu_{A}}$-almost surely in $\omega$, in the sense of weak convergence of measures.
Since the measure $L_t^\omega$ is random, there is some deviation to this convergence. We will study now the
rate of convergence.
In order to do it, we will prove the large deviation principle at level two for the{\it{\emph{ a priori }process}} $\{X_t,\,t\geq 0\}$.
We say there exists a large deviation principle at level two, if there exists a lower semicontinuous functional $I$, defined on
$\mathcal{M} (\{1,\dots,d\}^{\bb N})$,
such that:
{\it\begin{itemize}
 \item[i)] for any closed set $K\subset \mathcal{M} (\{1,\dots,d\}^{\bb N})$
$$ \limsup_{t \to \infty} \frac{1}{t} \, \log   \bb  P_{\mu_{A}}\big[ L_t \in K\big]\leq \,\,-\,\inf_{  \nu\in K} I(\nu)\,,$$

\item[ii)] for any open set $G\subset \mathcal{M} (\{1,\dots,d\}^{\bb N})$
$$ \liminf_{t \to \infty} \frac{1}{t} \, \log   \bb  P_{\mu_{A}}\big[ L_t\in G\big]\geq \,\,-\,\inf_{  \nu\in G} I(\nu)\,.$$
\end{itemize}}
We call $I$ the deviation function, or the rate function.

 \bigskip

In order to prove the result above, we observe that by the equality \eqref{eq.21}, we get
\begin{equation*}
 \begin{split}
e^{\,\,t\int_{\{1,\dots,d\}^{\bb N}} f(y)\,L_t^\omega(dy)}\,=\,e^{\,\int_0^t f(\omega_s)\,\mbox{d}s}\,,
\end{split}
\end{equation*}
for all $t\geq 0$,  $\omega\in\mc D$ and $f:\{1,\dots,d\}^{\bb N}\to \bb R$ bounded (or positive) mensurable function.
Then, we integrate both sides of the equality above concerning $\bb P_x$, and  we obtain
\begin{equation*}
 \begin{split}
\int_{\mc D}e^{\,\,t\int_{\{1,\dots,d\}^{\bb N}} f(y)\,L_t^\omega(dy)}\,\mbox{d}\bb P_x(\omega)\,=
\,\int_{\mc D}e^{\int_0^t f(\omega_s)\,\mbox{d}s}\,\mbox{d}\bb P_x(\omega)\,,
\end{split}
\end{equation*}
for all $t\geq 0$ and $f:\{1,\dots,d\}^{\bb N}\to \bb R$ bounded (or positive) mensurable function.
We recall that $\bb P_x$ is a probability on $\mc D$ induced by the
initial measure $\delta_x$ and the Markov process
$\{X_t;\, t\geq 0\}$.

Using   Lemma \ref{Q(V)}, and \eqref{lim} in  the previous section, and the last fact, we have
\begin{equation*}
 \begin{split}
  Q(V)&= \lambda_V=     \lim_{T \to \infty}\, \frac{1}{T}\,\log \int_{\mc D} e^{\int_0^{T} V(\omega_r)\,dr} \, \mbox{d} \bb  P_{\mu_{A}}(\omega) \\
&= \lim_{T \to \infty}\, \frac{1}{T}\,\log \int_{\mc D} e^{\,\,T\int_{\{1,\dots,d\}^{\bb N}} V(y)\,L_T^\omega(dy)}\, \mbox{d} \bb  P_{\mu_{A}}(\omega) .
 \end{split}
\end{equation*}
This shows that $Q(V)$ is the same one given in (1.3) of \cite{Ki1}, then this will allow us to find the functional rate $I$.
%
 From the general setting of \cite{Ki1} (there is no mention of eigenvalue in the below expression), we get
\begin{equation*}
0\leq I(\nu) =
\sup_{V\in \mc C}\, \,\Big(\int V \,\mbox{d} \nu\, -\, Q(V)\Big)
=\sup_{\at{V\in\mc C}{\text{and}\,\, V\,\,\text{is Lipschitz}} }\, \Big(\int V \,\mbox{d} \nu\, -\, Q(V)\Big)\,,
\end{equation*}
for any $\nu$ on $\mc P(\{1,\dots,d\}^{\bb N}) $ and $I(\mu)=\infty$ for all other $\mu\in \mc M(\{1,\dots,d\}^{\bb N})$.
We point out that the above expression for $I$ is in agreement with the one in Lemma \ref{F>0} by Lemma \ref{Q(V)}.
Since  the dual space of $\mc C$ is the space  $\mc M(\{1,\dots,d\}^{\bb N})$, we have
$$ Q(V)= \sup_{ \mu \in \mathcal{P}(\{1,\dots,d\}^{\bb N})}\, \Big(\int V \, \mbox{d} \mu - I(\mu)\Big)\,.$$

Following \cite{Ki1} we say that $\mu_V\in \mathcal{P} ( \{1,\dots,d\}^{\bb N} ) $ is an {\textit{equilibrium state for $V$}}, if
$$ Q(V)=  \int V \, \mbox{d} \mu_V - I(\mu_V)\,.$$


A  major result in the theory  is Theorem 2.1 in \cite{Ki1}. We will state a particular version of this result in Theorem \ref{teor23}.

\begin{theorem}\label{teor23} If for each Lipschitz function $V: \{1,\dots,d\}^{\bb N}  \to \mathbb{R}$ the equilibrium state $\mu_V$ is unique, then, the large deviation principle at  level two
is true with the deviation function
$$I(\nu) =\sup_{V\in \mc C}\, \,\Big(\int V \,\mbox{d} \nu\, -\, Q(V)\Big)\,.$$

\end{theorem}

\bigskip

From \eqref{lim} we get the upper bound estimate for $I$ and from Theorem \ref{teo20} (uniqueness) we get the lower bound estimate.
Then, we can state one of our main results (Theorem A in the Introduction):

\begin{theorem}
 Let $\{X_t, \,t \geq 0\}$ be the{\it{\emph{ a priori }process}},  then the large deviation principle at  level two
is true for our setting with the deviation function $I$ given by
$$I(\nu) =\sup_{V\in \mc C}\, \,\Big(\int V \,\mbox{d} \nu\, -\, Q(V)\Big)\,,$$
for any probability $\nu$ on $\{1,\dots,d\}^{\bb N} $ and $I(\nu)=\infty$ in any other case.

\end{theorem}

We point out that Lemma \ref{F>0} characterizes the equilibrium state in our setting.
We can state a major result due to Y. Kifer which follows by the reasoning of Section 4 in \cite{Ki1}.
This was adapted from the original claim.

\begin{theorem} \label{Ri}
If for each Lipschitz function $V: \{1,\dots,d\}^{\bb N}  \to \mathbb{R}$
there exists a  positive eigenfunction for the associated continuous time Ruelle operator,
 then, the deviation function $I$ is also given by
\begin{equation*}
I(\nu) =     - \inf_{u\in \mc C^+}\, \int \frac{L(u)}{u}\, \mbox{d} \nu\,.
\end{equation*}
\end{theorem}

It follows from last subsection (see Lemma \ref{F>0})  that the above expression is true in our setting. In this way our description of
the Large Deviation Principle at level two is completed.
We refer the reader to Lemma \ref{F>0} for
explicit expressions related to the above result.

We point out that the above Theorem \ref{Ri} in \cite{Ki1} (see also \cite{Ki2})
is presented in a different setting: the state space  is a  Riemannian manifold and it is considered a
certain class of differential operators as infinitesimal generators. We do not consider such differentiable structure. However, from last section we were able to adapt such reasoning to our setting.

\appendix

\section{The spectrum of ${\mc L}_A - I+V$ on ${\bb L}^2 (\mu_A)$ and Dirichlet form.}

For any $f\in \bb L^2 (\mu_A)$ the Dirichlet form of $f$ is
$${\mc E}_A\,(f,f) \,:=\, \<\,(I-{\mc L}_A)\, (f)\, , \, f\,\>_{\mu_A}\,. $$
Notice that
\begin{equation}\label{333}
 {\mc E}_A\,(f,f)= \frac{1}{2}\, \int \sum_{\sigma(y) =x} \, e^{A(y)} \,[f(x)-f(y)]^2\, \mbox{d} \, \mu_A (x)\geq 0\,.
\end{equation}
Indeed,
\begin{equation*}
 \begin{split}
  \<\,(I-{\mc L}_A)\, (f)\, , \, f\,\>_{\mu_A}\,= \,\int \,  \sum_{\sigma(y) =x} \, e^{A(y)} \,[f(x)-f(y)\big]\, f(x)\, \mbox{d} {\mu_A}(x)\,.
 \end{split}
\end{equation*}

By the other hand,
\begin{equation*}
 \begin{split}
 \<\,(I-{\mc L}_A)\, (f)\, , \, f\,\>_{\mu_A}\,&=\,\<f,f\>_{\mu_A}- \< {\mc L}_A(f),f\>_{\mu_A} \,=\,\int[{\mc L}_A(f^2)-{\mc L}_A(f)f]\,\mbox{d} {\mu_A} \,\\&=
  \, \int  \Big\{ \sum_{\sigma(y) =x} \, e^{A(y)} \,[f(y)-f(x)\big]\, f(y)\Big\}\, \mbox{d} {\mu_A}(x)\,.
 \end{split}
\end{equation*}
These two equalities imply that
\begin{equation*}
 \begin{split}
  \<\,(I-{\mc L}_A)\, (f)\, , \, f\,\>_{\mu_A}\,= \,\frac{1}{2}\, \int \sum_{\sigma(y) =x} \, e^{A(y)} \,[f(x)-f(y)]^2\, \mbox{d} \, \mu_A (x)\,.
 \end{split}
\end{equation*}

From expression \eqref{333} we have that ${\mc E}_A\,(f,f) =0$ implies  $f=0$.

We point out that we will consider bellow eigenvalues in ${\bb L}^2 (\mu_{A})$ which are not necessarily Lipschtiz.

Dirichlet forms are quite important (see \cite{kl}), among other reasons, because they are particulary useful when there is an spectral gap.
However, this will not be the case here.

\begin{proposition} Let a  Lipschitz function $V: \{1,\dots,d\}\to\bb R$ such that $\sup V-\inf V<2$.
There are eigenvalues $c$ for $\mc L_A-I+V$ in ${\bb L}^2 (\mu_{A})$ such that
 $\big[(\sup V-2)\vee 0\big]<c<\inf V$.  Each eigenvalue has infinite multiplicity.
Therefore, in this case, there is no spectral gap.
\end{proposition}

\begin{proof}
The existence of positive eigenvalues $c$
for the operator $\mc L_A-I+V$ satisfying  $\big[(\sup V-2)\vee 0\big]<c<\inf V$  will obtained
from solving the  twisted cohomological
equation.
In order to simplify the reasoning we will present the proof for the case $E=\{0,1\}^\mathbb{N}$.  From
section 2.2 in \cite{BS}, we know that given  functions $z:E\to \mathbb{R}$ and $C:E\to
\mathbb{R}$ one can solve in $\alpha$  the twisted cohomological
equation
\begin{equation}\label{solution}
 \frac{z(y)}{C(y)}= \frac{1}{C(y)}\alpha(y) - \alpha(\sigma(y)),
\end{equation}
in the case that $|C|<1$. Indeed,  just take
$$\alpha(y)= \,
 \sum_{j=0}^\infty \frac{\frac{z(\sigma^j(y))}{C(\sigma^j(y) )}}{(C(y)\, C(\sigma(y))\dots C(\sigma^j (y)))^{-1}}\,.$$
Note that this function $\alpha$ is measurable and bounded but not Lipschitz.

Take
$z(y)=(-1)^{y_0} e^{-A(y)}$, when $y=(y_0,y_1,y_2,\dots)$.
Now, for $c\in\big([(\sup V-2)\vee 0],\inf V\big)$ fixed, consider $C(y)= 1-V(\sigma(y))+c$. Notice that $|C|<1$.
Then, the equation \eqref{solution} becomes
\begin{equation*}
 (-1)^{y_0}= e^{A(y)}\Big\{\alpha(y) - \alpha(\sigma(y))\big(1-V(\sigma(y))+c\big)\Big\}\,.
\end{equation*}

Let $x\in\{1,\dots,d\}^{\bb N}$. Adding the equations above when $y=0x$ and when $y=1x$, we get
$$(\mc L_A-I+V)(\alpha)(x)\,=\, c\alpha(x)\,,$$
because $\sigma(0x)=x=\sigma(1x)$, and  the potential $A$ is normalized.

Is is also easy to show that changing a little bit the argument  one can get an infinite dimensional set
of possible $\alpha$ associated to the same eigenvalue.

\end{proof}

\section{Basic tools for continuous time Markov chains}\label{provalema}
In this section  we present the proofs of the Lemma \ref{pv} and Lemma \ref{lemma1}.
In order to do that, we will present another way to analyze the properties of a continuous time Markov chain.

Suppose the process $\{X_t,\,t\geq 0\}$ is a continuous time Markov chain.
In an alternative way we can described it by considering its skeleton chain (see \cite{Li} \cite{Pro}).
Let $\{\xi_n\}_{n\in\bb N}$ be a discrete time Markov chain
with transition probability given by $p(x,y)=\textbf{1}_{[\sigma(y)=x]}e^{A(y)}$.
Consider a sequence of random variables $\{\tau_n\}_{n\in \bb N}$,
which are independent and identically distributed according to an exponential
law  of parameter $1$.
For $n\geq 0$, define
\begin{equation*}
 T_0=0\,,\qquad\quad T_{n+1}=T_n+\tau_n=\tau_0+\tau_1+\dots+\tau_n\,.
\end{equation*}
Thus, $X_t$ can be rewritten as
$\sum_{n=0}^{+\infty} \xi_n \textbf{1}_{[T_n\leq t<T_{n+1}]}$, for all $t\geq 0$.

\begin{proof}[Proof of Lemma \ref{pv}]
Using the above,  we are able to describe  expression \eqref{0} in a different way:
\begin{equation*}
\begin{split}
 &P_{T}^V (f)(x)\,=\, \bb E_{x} \big[e^{\int_0^{T} V(X_r)\,dr} f(X_{T})\big]
 =\sum_{n=0}^{+\infty}
 \bb E_{x} \big[e^{\int_0^{T} V(X_r)\,dr} f(X_{T})\textbf{1}_{[T_n\leq T< T_{n+1}]}\big]\\
 &= \sum_{n=0}^{+\infty}
 \bb E_{x} \big[e^{T_1 V(\xi_0)+ (T_2-T_1)V(\xi_1)+\dots+(T_n-T_{n-1}) V(\xi_{n-1})+(T-T_n) V(\xi_{n})}
 f(\xi_n)\textbf{1}_{[T_n\leq T< T_{n+1}]}\big]\\
  &= \sum_{n=0}^{+\infty}
 \bb E_{x} \big[e^{\tau_0 V(\xi_0)+\tau_1 V(\xi_1)+\dots+\tau_{n-1} V(\xi_{n-1})+(T-\sum_{i=0}^{n-1}\tau_i) V(\xi_{n})}
 f(\xi_n)\textbf{1}_{[\sum_{i=0}^{n-1}\tau_i\leq T< \sum_{i=0}^{n}\tau_i]}\big]\\
&= \bb E_{x} \big[e^{T V(\xi_0)}
 f(\xi_0)\textbf{1}_{[ T< \tau_0]}\big]\,\,+\\
&
\sum_{n=1}^{+\infty} \sum_{a_1=1}^{d}\dots \sum_{a_n=1}^{d}
 \bb E_{x} \big[e^{\tau_0 V(\xi_0)+\dots+(T-\sum_{i=0}^{n-1}\tau_i) V(\xi_{n})}
 f(\xi_n)\textbf{1}_{[\sum_{i=0}^{n-1}\tau_i\leq T< \sum_{i=0}^{n}\tau_i]}\textbf{1}_{[\xi_{1}=a_1x,\dots,\xi_{n}=a_n\dots a_1x]}\big]\,,
\end{split}
\end{equation*}
where $\sigma^n(a_n\dots a_1x)=x$.
The first term above is equal to
 $e^{T V(x)}
 f(x)e^{-T}$.
The summand in the second one is equal to
\begin{equation*}
\begin{split}
  \bb E_{x} \Big[&e^{\tau_0 V(\xi_0)+\dots+(T-\sum_{i=0}^{n-1}\tau_i) V(\xi_{n})}
 f(\xi_n)\textbf{1}_{[\sum_{i=0}^{n-1}\tau_i\leq T< \sum_{i=0}^{n}\tau_i]}\Big|\xi_{1}=a_1x,\dots,\xi_{n}=a_n\dots a_1x\Big]\cdot\\
&\cdot\,
\bb P_{x} \big[\xi_{1}=a_1x,\dots,\xi_{n}=a_n\dots a_1x\big]\,.
\end{split}
\end{equation*}
Using the transition probability of the Markov chain $\{\xi_n\}_n$, we get
\begin{equation*}
 \bb P_{x} \big[\xi_{1}=a_1x,\dots,\xi_{n}=a_n\dots a_1x\big]\,=\,e^{A(a_1x)}\dots\, e^{A(a_n\dots a_1x)}\,.
\end{equation*}
Recalling that the random variables $\{\tau_i\}$ are independent and identically distributed according to an exponential
law  of parameter $1$, we have
\begin{equation*}
\begin{split}
 & \bb E_{x} \Big[e^{\tau_0 V(\xi_0)+\dots+(T-\sum_{i=0}^{n-1}\tau_i) V(\xi_{n})}
 f(\xi_n)\textbf{1}_{[\sum_{i=0}^{n-1}\tau_i\leq T< \sum_{i=0}^{n}\tau_i]}\Big|\xi_{1}=a_1x,\dots,\xi_{n}=a_n\dots a_1x\Big]\\
&= \bb E_{x} \Big[e^{\tau_0 V(x)+\dots+(T-\sum_{i=0}^{n-1}\tau_i) V(a_n\dots a_1x)}
 f(a_n\dots a_1x)\textbf{1}_{[\sum_{i=0}^{n-1}\tau_i\leq T< \sum_{i=0}^{n}\tau_i]}\Big]\\
&= f(a_n\dots a_1x)\! \int_0^\infty \!\!\!\!\!dt_{n} \dots \!\!\int_0^\infty\!\!\! \!\!dt_{0}\,
e^{t_0 V(x)+\dots+(T-\sum_{i=0}^{n-1}t_i) V(a_n\dots a_1x)}
\textbf{1}_{[\sum_{i=0}^{n-1}t_i\leq T< \sum_{i=0}^{n}t_i]}e^{-t_0}\!\dots e^{-t_n}\,.
\end{split}
\end{equation*}
Therefore,
\begin{equation*}
\begin{split}
 &P_{T}^V (f)(x)\,=\, \bb E_{x} \big[e^{\int_0^{T} V(X_r)\,dr} f(X_{T})\big]=e^{T V(x)}
 f(x)e^{-T}\,+\\
&
\sum_{n=1}^{+\infty} \sum_{a_1=1}^{d}\dots \sum_{a_n=1}^{d}
e^{A(a_1x)}\dots\, e^{A(a_n\dots a_1x)}
f(a_n\dots a_1x) \cdot\\
&\int_0^\infty \!\!\!\!dt_{n} \dots \int_0^\infty\!\! \!\!dt_{0}\,\,
e^{t_0 V(x)+\dots+(T-\sum_{i=0}^{n-1}t_i) V(a_n\dots a_1x)}
\textbf{1}_{[\sum_{i=0}^{n-1}t_i\leq T< \sum_{i=0}^{n}t_i]}e^{-t_0}\dots e^{-t_n}\,.
\end{split}
\end{equation*}
\end{proof}

\begin{proof}[Proof of Lemma \ref{lemma1}]
We begin analyzing
\begin{equation}\label{Iv}
\begin{split}
&\mc I_V^T(a_n\dots a_1 x)\\
&=\int_0^\infty \!\!\!\!dt_{n} \dots \int_0^\infty\!\! \!\!dt_{0}\,\,
e^{t_0 V(x)+\dots+(T-\sum_{i=0}^{n-1}t_i) V(a_n\dots a_1x)}
\textbf{1}_{[\sum_{i=0}^{n-1}t_i\leq T< \sum_{i=0}^{n}t_i]}e^{-t_0}\dots e^{-t_n}\\
&\leq e^{TC_Vd(x,y)+TC_Vd(a_1x,a_1y)+\dots+TC_Vd(a_n\dots a_1x,a_n\dots a_1y)}\cdot\\
&\cdot\int_0^\infty \!\!\!\!dt_{n} \dots \int_0^\infty\!\! \!\!dt_{0}\,\,
e^{t_0 V(y)+\dots+(T-\sum_{i=0}^{n-1}t_i) V(a_n\dots a_1y)}
\textbf{1}_{[\sum_{i=0}^{n-1}t_i\leq T< \sum_{i=0}^{n}t_i]}e^{-t_0}\dots e^{-t_n}\\
&\leq e^{TC_V(1+\theta+\dots+\theta^n)d(x,y)}\cdot\\
&\cdot\int_0^\infty \!\!\!\!dt_{n} \dots \int_0^\infty\!\! \!\!dt_{0}\,\,
e^{t_0 V(y)+\dots+(T-\sum_{i=0}^{n-1}t_i) V(a_n\dots a_1y)}
\textbf{1}_{[\sum_{i=0}^{n-1}t_i\leq T< \sum_{i=0}^{n}t_i]}e^{-t_0}\dots e^{-t_n}\\
&\leq e^{TC_V(1-\theta)^{-1}d(x,y)}\mc I_V^T(a_n\dots a_1 y)\,
\end{split}
\end{equation}
and $e^{T V(x)}
 e^{-T}\leq  e^{TC_Vd(x,y)} e^{T V(y)}
e^{-T}$.
Since the potential $A$  is also Lipschitz, we get
\begin{equation}\label{Alip}
\begin{split}
   e^{A(a_1x)}\dots\, e^{A(a_n\dots a_1x)}&\leq
 e^{C_A(\theta+\dots+\theta^n)d(x,y)}
e^{A(a_1y)}\dots\, e^{A(a_n\dots a_1y)}\\
&\leq
 e^{C_A\theta(1-\theta)^{-1}d(x,y)}
e^{A(a_1y)}\dots\, e^{A(a_n\dots a_1y)}\,.
                 \end{split}
\end{equation}
By the hypothesis we assume for $f$, we get
$$f(a_n\dots a_1x)\leq e^{C_f\theta^n d(x,y)} f(a_n\dots a_1y)\leq e^{C_f\theta d(x,y)} f(a_n\dots a_1y)\,.$$
Thus,
\begin{equation*}
\begin{split}
&P_{T}^V (f)(x)\,=\,e^{T V(x)}
 e^{-T}\,+\,
\sum_{n=1}^{+\infty} \sum_{a_1=1}^{d}\!\dots\! \sum_{a_n=1}^{d}
e^{A(a_1x)}\dots e^{A(a_n\dots a_1x)} f(a_n\dots a_1x)
 \mc I_V^T(a_n\dots a_1 x)\\
&\leq e^{TC_Vd(x,y)} e^{T V(y)}
e^{-T}\\
&+e^{[(C_A\theta+TC_V)(1-\theta)^{-1}+C_f\theta]d(x,y)}\sum_{n=1}^{+\infty} \!\sum_{a_1=1}^{d}\!\dots\! \sum_{a_n=1}^{d}
e^{A(a_1y)}\!\dots e^{A(a_n\dots a_1y)}f(a_n\dots a_1y)
\mc I_V^T(a_n\dots a_1 y)\\
&\leq e^{[(C_A\theta+TC_V)(1-\theta)^{-1}+C_f\theta]d(x,y)} \Big[ e^{T V(y)}
e^{-T}\\&+\sum_{n=1}^{+\infty} \!\sum_{a_1=1}^{d}\!\dots\! \sum_{a_n=1}^{d}
e^{A(a_1y)}\!\dots e^{A(a_n\dots a_1y)} f(a_n\dots a_1y)
\mc I_V^T(a_n\dots a_1 y)\Big]\\
&\leq e^{[(C_A\theta+TC_V)(1-\theta)^{-1}+C_f\theta]d(x,y)}P_{T}^V (f)(y)\,.
\end{split}
\end{equation*}

\end{proof}

\section{Radon-Nikodym derivative}\label{RN}
Let $\{\mc F_T,\,T\geq 0\}$ be the natural filtration.
\begin{proposition}\label{apendice1}
  The Radon-Nikodim derivative of the measure $\bb P_\mu$ (associated to the{\it{\emph{ a priori }process}})
    concerning the admissible measure $\tilde{\bb P}_\mu$ (see Definition \ref{admissible}) restricted to $\mc F_T$ is
  \begin{equation*}
 \begin{split}
  \frac{\mbox{d}\bb P_\mu}{\mbox{d}\tilde{\bb P}_\mu}\Big|_{\mc F_T}
  =\exp\Bigg\{\int_0^T (\tilde{\gamma}(X_s)-1)\,\mbox{d}s+\sum_{s\leq T}\textbf{1}_{[\sigma(X_s)=X_{s^-}]}
\Big(A(X_s)-\tilde{A}(X_s)-\log\big(\tilde{\gamma}(\sigma(X_s))\big)\Big)\Bigg\}\,.
 \end{split}
\end{equation*}
\end{proposition}

\begin{proof}
The probabilities  $\tilde{\bb P}_\mu$ and  $\bb P_\mu$  on $\mc D$ are equivalent, because the initial measure and the allowed jumps are the same.
 Thus, the expectation under $\bb E_{\mu}$ of all bounded function $\psi:\mc D\to \bb R$, $\mc F_T$-measurable,
is
\begin{equation*}
 \begin{split}
\tilde{\bb E}_{\mu}\Big[\,\psi\,
\frac{\mbox{d}\bb P_\mu}{\mbox{d}\tilde{\bb P}_\mu}\Big|_{\mc F_T}\Big]\,.
 \end{split}
\end{equation*}
The goal here is to obtain a formula for  the Radon-Nikodim derivative $\frac{\mbox{d}\bb P_\mu}{\mbox{d}\tilde{\bb P}_\mu}$.
Since every bounded $\mc F_T$-measurable function can be approximated by functions depending only on a finite number of coordinates,
then, it is enough to work with these functions.
For $k\geq 1$, consider a sequence of times $0\leq t_1<\dots<t_k\leq T$ and a bounded function $F:\big(\{1,\dots,d\}^{\bb N}\big)^k\to\bb R$.
Using the skeleton chain, presented in the proof of  Lemma \ref{pv}, we get
\begin{equation*}
 \begin{split}
\bb E_{\mu}[F(X_{t_1},\dots,X_{t_k})]=\sum_{n\geq 0}
\bb E_{\mu}\big[F(X_{t_1},\dots,X_{t_k})\textbf{1}_{[T_n\leq T< T_{n+1}]}\big]
\,.
 \end{split}
\end{equation*}
Since $F(X_{t_1},\dots,X_{t_k})$ restricted to the set $[T_n\leq T< T_{n+1}]$
 depends only on $\xi_1, T_1,\dots, \xi_n,T_n$, there exist functions $\bar F_n$ such that
 \begin{equation*}
 \begin{split}
\bb E_{\mu}[F(X_{t_1},\dots,X_{t_k})]=\sum_{n\geq 0}
\bb E_{\mu}\big[\bar F_n(\xi_1, T_1,\dots, \xi_n,T_n)\textbf{1}_{[T_n\leq T< T_{n+1}]}\big]
\,.
 \end{split}
\end{equation*}
Through some calculations that are similar to the one used on the Corollary 2.2 in Appendix 1 of the \cite{kl},
 the last probability is equal to
\begin{equation}\label{eq22}
 \begin{split}
\sum_{n\geq 0}
\bb E_{\mu}\big[\bar F_n(\xi_1, T_1,\dots, \xi_n,T_n)\,\textbf{1}_{[T_n\leq T]}\,\,e^{-\lambda(\xi_n)(T-T_n)}\big]\,.
 \end{split}
\end{equation}

Then, we need to estimate
for each $n\in \bb N$ and, moreover, for all bounded measurable function  $G: \big(\{1,\dots,d\}^{\bb N}\times (0,\infty)\big)^n\to\bb R$ the expectation
\begin{equation*}
 \begin{split}
\bb E_{\mu}\big[G(\xi_1, T_1,\dots, \xi_n,T_n)\big]=\int_{\{1,\dots,d\}^{\bb N}}\bb E_{x}\big[G(\xi_1, T_1,\dots, \xi_n,T_n)\big]
\,\mbox{d}\mu(x)\,.
 \end{split}
\end{equation*}
Notice that, for all $x\in \{1,\dots,d\}^{\bb N}$,
\begin{equation*}
 \begin{split}
&\bb E_{x}\big[G(\xi_1, T_1,\dots, \xi_n,T_n)\big]= \sum_{a_1=1}^{d}\dots \sum_{a_n=1}^{d}
e^{A(a_1x)}\dots\, e^{A(a_n\dots a_1x)}\\
&\qquad\quad\qquad\quad \cdot\,\,\Big\{\int_0^\infty \!\!\!\!dt_{n-1} \dots \int_0^\infty\!\! \!\!dt_{0}\,\,
e^{-t_0}\dots e^{-t_{n-1}}\,G(a_1x,t_0,\cdots,a_n\dots a_1 x,t_{n-1}+\dots+t_0)\Big\}\\
&= \sum_{a_1=1}^{d}\dots \sum_{a_n=1}^{d}
e^{\tilde A(a_1x)}\dots\, e^{\tilde A(a_n\dots a_1x)}\Bigg\{\!\!\int_0^\infty \!\!\!\!dt_{n-1} \dots \int_0^\infty\!\! \!\!dt_{0}\,\,
\tilde\gamma(x)e^{-\tilde\gamma(x)t_0}\dots \tilde\gamma(a_{n-1}\dots a_1x)e^{-\tilde\gamma(a_{n-1}\dots a_1x)t_{n-1}}\\
&\qquad\quad\qquad\quad \cdot\,\,e^{A(a_1x)-\tilde A(a_1x)}\dots\, e^{A(a_n\dots a_1x)-\tilde A(a_n\dots a_1x)}\,
\,\frac{e^{(\tilde\gamma(x)-1)t_0}}{\tilde\gamma(x)}\dots \frac{e^{(\tilde\gamma(a_{n-1}\dots a_1x)-1)t_{n-1}}}{\tilde\gamma(a_{n-1}\dots a_1x)}\\
&\qquad\quad\qquad\quad \cdot\,\,G(a_1x,t_0,\cdots,a_n\dots a_1 x,t_{n-1}+\dots+t_0)\,\Bigg\}\\
&=\tilde{\bb E}_{x}\Big[G(\xi_1, T_1,\dots, \xi_n,T_n)\,\exp\Big\{\sum_{i=0}^{n-1}(\tilde\gamma(\xi_i)-1)\tau_i\Big\}\prod_{i=0}^{n-1}
\,e^{A(\xi_{i+1})-\tilde A(\xi_{i+1})}\,\pfrac{1}{\tilde\gamma(\xi_i)}\Big]\,.
\end{split}
\end{equation*}

We can write $\sum_{i=0}^{n-1}(\tilde\gamma(\xi_i)-1)\tau_i$ as
\begin{equation*}
\sum_{i=0}^{n-1}(\tilde\gamma(\xi_i)-1)\int_{0}^{T_n} \textbf{1}_{[T_i\leq s< T_{i+1}]}\,\mbox{d}s=
\int_{0}^{T_n}\sum_{i=0}^{\infty}(\tilde\gamma(\xi_i)-1) \,\textbf{1}_{[T_i\leq s< T_{i+1}]}\,\mbox{d}s=
\int_{0}^{T_n}(\tilde\gamma(X_s)-1) \,\mbox{d}s,
\end{equation*}
and, we can write $e^{A(\xi_{i+1})-\tilde A(\xi_{i+1})}\,\pfrac{1}{\tilde\gamma(\xi_i)}$ as
\begin{equation*}
 \begin{split}
 & \exp\Big\{\sum_{i=0}^{n-1}(A(\xi_{i+1})-\tilde A(\xi_{i+1})-\log\tilde\gamma(\xi_i))\Big\}\\=&
  \exp\Big\{\sum_{i=0}^{n-1}\textbf{1}_{[\sigma(\xi_{i+1})=\xi_{i}]}\big(A(\xi_{i+1})-\tilde A(\xi_{i+1})-\log\tilde\gamma(\sigma(\xi_{i+1}))\big)\Big\}\\
   =&\exp\Big\{\sum_{s\leq T_n}\textbf{1}_{[\sigma(X_s)=X_{s^-}]}
\Big(A(X_s)-\tilde{A}(X_s)-\log\big(\tilde{\gamma}(\sigma(X_s))\big)\Big)\Big\}\,.
 \end{split}
\end{equation*}
The expectation under $\bb P_{x}$ of $G(\xi_1, T_1,\dots, \xi_n,T_n)$ becomes
\begin{equation*}
 \begin{split}
\tilde{\bb E}_{x}\Big[G(\xi_1, T_1,\dots, \xi_n,T_n)\,
\exp\Big\{\int_{0}^{T_n}(\tilde\gamma(X_s)-1) \,\mbox{d}s+\sum_{s\leq T_n}\textbf{1}_{[\sigma(X_s)=X_{s^-}]}
\Big(A(X_s)-\tilde{A}(X_s)-\log\big(\tilde{\gamma}(\sigma(X_s))\big)\Big)\Big\}\Big]\,.
\end{split}
\end{equation*}
Using the formula above in the equation \eqref{eq22}, the expectation under $\bb E_{\mu}$ of  $F(X_{t_1},\dots,X_{t_k})$ is equal to
\begin{equation*}
 \begin{split}
&\sum_{n\geq 0}
\tilde{\bb E}_{\mu}\Bigg[\bar F_n(\xi_1, T_1,\dots, \xi_n,T_n)\,
\textbf{1}_{[T_n\leq T]}\,\,e^{-\lambda(\xi_n)(T-T_n)}\\
&\quad\cdot\,\,\exp\Big\{\int_{0}^{T_n}(\tilde\gamma(X_s)-1) \,\mbox{d}s+\sum_{s\leq T_n}\textbf{1}_{[\sigma(X_s)=X_{s^-}]}
\Big(A(X_s)-\tilde{A}(X_s)-\log\big(\tilde{\gamma}(\sigma(X_s))\big)\Big)\Big\}\Bigg]\,.
 \end{split}
\end{equation*}
 Once again, we use some calculations similarly to the Corollary 2.2 in Appendix 1 of the \cite{kl} and we rewrite the expression above as
 \begin{equation*}
 \begin{split}
&\sum_{n\geq 0}
\tilde{\bb E}_{\mu}\Bigg[\bar F_n(\xi_1, T_1,\dots, \xi_n,T_n)\,
\textbf{1}_{[T_n\leq T< T_{n+1}]}\\
&\quad\cdot\,\,\exp\Big\{\int_{0}^{T}(\tilde\gamma(X_s)-1) \,\mbox{d}s+\sum_{s\leq T}\textbf{1}_{[\sigma(X_s)=X_{s^-}]}
\Big(A(X_s)-\tilde{A}(X_s)-\log\big(\tilde{\gamma}(\sigma(X_s))\big)\Big)\Big\}\Bigg],
 \end{split}
\end{equation*}
and, this sum is equal to
 \begin{equation*}
 \begin{split}
&
\tilde{\bb E}_{\mu}\Big[F(X_{t_1},\dots,X_{t_k})\,
\exp\Big\{\int_{0}^{T}(\tilde\gamma(X_s)-1) \,\mbox{d}s+\sum_{s\leq T}\textbf{1}_{[\sigma(X_s)=X_{s^-}]}
\Big(A(X_s)-\tilde{A}(X_s)-\log\big(\tilde{\gamma}(\sigma(X_s))\big)\Big)\Big\}\Big]\,.
 \end{split}
\end{equation*}
This finish the proof.

\end{proof}

\section{Proof of Lemma \ref{lemma11}}\label{provalema11}

\begin{proof}[Proof of Lemma \ref{lemma11}]
We claim that
\begin{equation*}
\begin{split}
 M^G_T(\omega)=\sum_{s\leq T}\textbf{1}_{\{\sigma(\omega_s)\,=  \,\omega_{s^-}\}} G(\omega_s)
\,
- \,\int_0^T\tilde{\gamma}(\omega_s)G(\omega_s)\, \mbox{d}s
\end{split}
\end{equation*}
is a $\tilde{\bb P}_\mu$ - martingale. Then, this lemma will follow from $\tilde{\bb E}_\mu\big[M^G_T\big]=\tilde{\bb E}_\mu\big[M^G_0\big]=0$.
In order to prove this claim it is enough to prove that
\begin{equation}\label{mart}
\begin{split}
 M_T(\omega)=\sum_{s\leq T}\textbf{1}_{\{\sigma(\omega_s)\,=  \,\omega_{s^-}\}}
\,
- \,\int_0^T\tilde{\gamma}(\omega_s)\, \mbox{d}s
\end{split}
\end{equation}
is a $\tilde{\bb P}_\mu$ - martingale, because $ M^G_T=\int G\,\mbox{d}M_T$ will be a $\tilde{\bb P}_\mu$ - martingale (see \cite{ry}).

Now, we prove \eqref{mart}.
Let $\{\mc F_T,\,T\geq 0\}$ be the natural filtration. For all $S<T$, we  prove that
$
\tilde{\bb E}_\mu\big[M_T-M_S\vert \mc F_S\big]=0.
$
By Markov
property, we only need to show that
$
\tilde{\bb E}_x\big[M_t\big]=0.
$

Denote by  $\mc D_x $ the space of  all trajectories $\omega$ in $\mc D$ such that $\omega_0=x$.
Observe that,  for all $\omega$ in $\mc D_x$,
\begin{equation}\label{11}
\begin{split}
 \int_0^t\tilde{\gamma}(\omega_s)\, \mbox{d}s\,=\, \sum_{k\geq 1}\sum_{i_1=1}^{d}\cdots\sum_{i_k=1}^{d}\tilde{\gamma}(i_k\dots i_1x)
\int_0^t\textbf{1}_{[\omega_s=i_k\dots i_1x]}\,\mbox{d}s\,.
\end{split}
\end{equation}

For all $s\geq 0$ and $y \in \{1,\dots,d\}^{\bb N}$, $N_s(y)$ denotes the number of times that the exponential clock rang at site $y$.
Thus, the first term on the right side of \eqref{mart} can be rewritten as
\begin{equation}\label{111}
\begin{split}
\sum_{s\leq t}\textbf{1}_{\{\sigma(\omega_s)\,=  \,\omega_{s^-}\}} =  \sum_{k\geq 1}\sum_{i_1=1}^{d}\cdots\sum_{i_k=1}^{d} N_t(i_k\dots i_1x)\,,
\end{split}
\end{equation}
for all $\omega$ in $\mc D_x$.

Since \eqref{11} and \eqref{111} are true, in order to conclude this prove, it is sufficient  to show that
\begin{equation}\label{ex}
\begin{split}
 \tilde{\bb E}_x\big[N_t(y)-\tilde{\gamma}(y)
\int_0^t\textbf{1}_{[X_s=y]}\,\mbox{d}s\big]\,= \,0\,,
\end{split}
\end{equation}
for all $y \in \{1,\dots,d\}^{\bb N}$.

Let $0=t_0<t_1<\cdots<t_n=t$ be a partition of the interval $[0,t]$. The expression \eqref{ex} can be rewritten as
\begin{equation*}
 \sum_{i=0}^{n-1}\tilde{\bb E}_x\Big[ N_{t_{i+1}}(y)-N_{t_{i}}(y)+\tilde{\gamma}(y)\int_{t_i}^{t_{i+1}} \textbf{1}_{[X_s=y]}\,\mbox{d}s\Big]\,.
\end{equation*}
Observe that
\begin{equation*}
\begin{split}
& \tilde{\bb E}_x\Big[ \int_{t_i}^{t_{i+1}} \textbf{1}_{[X_s=y]}\,\mbox{d}s\Big]
=\tilde{\bb E}_y\Big[ \int_{0}^{t_{i+1}-t_i} \textbf{1}_{[X_s=y]}\,\mbox{d}s\Big]\\
=&\tilde{\bb E}_y\Big[ \int_{0}^{t_{i+1}-t_i} \textbf{1}_{[X_s=y]}\,\mbox{d}s\, \textbf{1}_{[N_{t_{i+1}-t_i}(y)=0]}\Big]
+\tilde{\bb E}_y\Big[ \int_{0}^{t_{i+1}-t_i} \textbf{1}_{[X_s=y]}\,\mbox{d}s\, \textbf{1}_{[N_{t_{i+1}-t_i}(y)>0]}\Big]\\
=&(t_{i+1}-t_i)+O_{\tilde\gamma}\big((t_{i+1}-t_i)^2\big)\,,
\end{split}
\end{equation*}
where the function $O_{\tilde{\gamma}}$ satisfies $O_{\tilde{\gamma}}(h)\leq C_{\tilde{\gamma}} h$.
 Then, we only need to prove that
\begin{equation*}\label{exp2}
 \tilde{\bb E}_x\big[ N_{t_{i+1}}(y)-N_{t_{i}}(y)\big]=\tilde{\gamma}(y)(t_{i+1}-t_i)\,.
\end{equation*}
By the Markov Property, it is enough to see  that $ \tilde{\bb E}_x[ N_{h}(y)]=\tilde{\gamma}(y)h$.
This is a consequence of the  $\tilde{\gamma}(y)$ being the parameter of the exponential clock at the site $y$.
\end{proof}

\section{Basic properties of $Q(V)$}\label{apQV}

\begin{lemma}
 $\vert Q(V)-Q(U)\vert\leq \Vert V-U\Vert_{\infty}$.
\end{lemma}
\begin{proof}Since
\begin{equation*}
\begin{split}
 & P_{T}^{V}(1)(x)\,=\,\bb E_x\Big[e^{\int_0^T V(X_r)\,dr}\Big]
\,\leq\,\bb E_x\Big[ e^{T \Vert V-U\Vert_{\infty}}e^{\int_0^T U(X_r)\,dr}\Big]\,=\, e^{T \Vert V-U\Vert_{\infty}} P_{T}^{U}(1)(x)\,,
\end{split}
\end{equation*}
then,
\begin{equation*}
 \begin{split}
&\vert Q(V)-Q(U)\vert=\lim_{T \to \infty}\, \frac{1}{T}\,\log
\frac{\int P_{T}^V(1)(x)\,\mbox{d} \mu_A(x) }{\int P_{T}^U(1)(x)\,\mbox{d} \mu_A(x) }  \\
&\leq\,\lim_{T \to \infty}\, \frac{1}{T}\,\log\frac{\int e^{T \Vert V-U\Vert_{\infty}} (P_{T}^U 1)(x)\,\mbox{d} \mu_A(x) }{\int P_{T}^U(1)(x)\,\mbox{d} \mu_A(x) }  \\
& =\, \Vert V-U\Vert_{\infty}\,.
\end{split}
\end{equation*}
\end{proof}

\begin{lemma}
 The functional  $V\to Q(V)$ is convex, i.e.,
for all $\alpha\in(0,1)$, we have
\begin{equation*}
 Q(\alpha V+(1-\alpha)U)\leq\alpha Q(V)+(1-\alpha)Q(U)\,.
\end{equation*}
\end{lemma}
\begin{proof}
Using the Holder's inequality, we have
\begin{equation*}
\begin{split}
 &\int P_{T}^{\alpha V+(1-\alpha)U}(1)(x) \,\mbox{d} \mu_A(x) \,=\,\bb E_{\mu_A}\Big[e^{\int_0^T \alpha V(X_r)\,dr}e^{\int_0^T (1-\alpha) U(X_r)\,dr}\Big]\\
&\,\leq\,\Big(\bb E_{\mu_A}\Big[e^{\int_0^T V(X_r)\,dr}\Big]\Big)^{\alpha }\Big(\bb E_{\mu_A}\Big[e^{\int_0^T  U(X_r)\,dr}\Big]\Big)^{(1-\alpha)}\,.
\end{split}
\end{equation*}
Thus,
 \begin{equation*}
 \begin{split}
Q(\alpha V+(1-\alpha)U)\,=\,&\lim_{T \to \infty}\, \frac{1}{T}\,\log
\int P_{T}^{\alpha V+(1-\alpha)U}(1)(x) \,\mbox{d} \mu_A(x)  \\
\leq\,&\lim_{T \to \infty}\, \frac{1}{T}\,\log
 \Big(\int\bb E_{\mu_A}\Big[e^{\int_0^T V(X_r)\,dr}\Big]\Big)^{\alpha }\\
 &\,\,\times
\Big(\int\bb E_{\mu_A}\Big[e^{\int_0^T  U(X_r)\,dr}\Big]\Big)^{(1-\alpha)}\\
=\,&\alpha\lim_{T \to \infty}\, \frac{1}{T}\,\log
 \int\bb E_x\Big[e^{\int_0^T V(X_r)\,dr}\Big] \,\mbox{d} \mu_A(x) \\
 &+(1-\alpha)
\lim_{T \to \infty}\, \frac{1}{T}\,\log\int\bb E_x\Big[e^{\int_0^T  U(X_r)\,dr}\Big]\,\mbox{d} \mu_A(x) \,.
 \end{split}
\end{equation*}
\end{proof}

\section{The associated symmetric process and the Metropolis algorithm}\label{end}

We can consider in our setting an extra parameter $\beta\in \mathbb{R}$ which plays the role of the inverse of temperature.
For a given fixed potential $V$ we can consider the new potential $\beta V$, $\beta \in\mathbb{R}$, and applying what we did before, we
get continuous time equilibrium states described by $\gamma_\beta:=\gamma_{\beta V}$ and $B_\beta:=B_{\beta V}$, in the previous notation.
In other words, we consider the infinitesimal generator $(\mathcal{L}_A-I)+ \beta V$, $\beta >0$, and the associated
main eigenvalue $\lambda_\beta:=\lambda_{\beta V}$.
We denote by $L^{V,\beta}$ the infinitesimal generator of the process that is the continuous time Gibbs state for the
potential $\beta V$, then $L^{V,\beta}$ acts on functions $f$ as
$ L^{V,\beta}(f)(x)= \gamma_{\beta }(x)\, \sum_{\sigma(y)=x}e^{B_{\beta }(y)}\big[f(y)-f(x)\big]$.
We are interested in the stationary probability
$\mu_\beta:=\mu_{B_{\beta V}, \gamma_{\beta V}}$ for the semigroup $\{e^{\,t\,L^{V,\beta} }, \,t\geq 0\}$,
and its weak limit as $\beta \to \infty.$
 This limit would correspond to the continuous time Gibbs state for temperature zero
(see \cite{CLT}, \cite{LMST} and \cite{LMMS} for related results).

The dual of $L^{V,\beta} $ on the Hilbert space $\bb L^2( \mu_{\beta})$ is  ${L^{V,\beta}}^*=\gamma_\beta\, (\mathcal{K}-I)$, where
$\mathcal{K}$ is the Koopman operator.
Notice that the probability $\mu_{\beta}$ is also stationary for the continuous time process with symmetric infinitesimal generator
 $L_{sym}^{V,\beta}:=\frac{1}{2}  (L^{V,\beta} + {L^{V,\beta}}^*).$
In this new process the particle at $x$ can jump to a $\sigma-$preimage $y$ with probability $\frac{1}{2}e^{B_{\beta}(y)}$,
or with probability $\frac{1}{2}$, to the forward image $\sigma(x)$, but, in both ways,   according to a exponential time of parameter $\gamma_\beta(x)$.

The eigenfunction of the continuous time Markov chain with infinitesimal generator
$L_{sym}^{V,\beta}$ can be different from the one with generator $L^{V,\beta}$.
Given $V$ and $\beta$, we denote $\lambda(\beta)_{sym} $ the main eigenvalue that we obtained
from $\beta \,V$ and the generator  $L_{sym}^{V,\beta}$.
The eigenvalues of
${L^{V,\beta}}$ and  ${L^{V,\beta}}^*$ are the same as before.
Now, we will look briefly at how to obtain $\lambda(\beta)_{sym}$.
From the symmetric assumption and  \cite{DS}, we get, for a fixed $\beta$,
\begin{equation*}
 \begin{split}
  \lambda(\beta)_{sym}\,=\,&\sup_{\at{ \phi\in \bb L^2(\mu_{\beta}),}{ \Vert\phi\Vert_2=1}}
\, \int \phi^{1/2}\,\Big[\frac{\gamma_\beta }{2}\big( [  \mathcal{L}_\beta + \mathcal{K}]- 2I\big)\,+\, \beta V\Big] (\phi^{1/2})
 \,  \mbox{d} \mu_{\beta}\\
=\,& \sup_{\at{ \phi\in \bb L^2(\mu_{\beta}),}{ \Vert\phi\Vert_2=1}}  \, \int \phi^{1/2}
\,\Big[ \frac{1}{2}\big( [  \mathcal{L}_\beta + \mathcal{K}]-2I\big)\,+\,\frac{1}{\gamma_\beta}\,\beta V\Big] (\phi^{1/2})   \,
\frac{\mbox{d} \mu_{B_{\beta}}}{\int \frac{1}{\gamma_\beta}\, \mbox{d} \mu_{B_{\beta}}}\\
=\,&\sup_{\at{ \phi\in \bb L^2(\mu_{\beta}),}{ \Vert\phi\Vert_2=1}}  \, \int
\Big\{ \phi^{1/2}  \mathcal{L}_\beta (\phi^{1/2})\,-1\,+\,\frac{1}{ \gamma_\beta}\,\beta V |\phi| \Big\}
\frac{\mbox{d}\mu_{B_{\beta}}}{\int \frac{1}{\gamma_\beta}\, \mbox{d} \mu_{B_{\beta}}}\,.
 \end{split}
\end{equation*}
The second equality is due to the Definition \eqref{muVgamma}, and
 the last one is by the dual, $\mathcal{L}_\beta^*$, on $\bb L^2(\mu_{\beta})$ is $\mathcal{K}$.

Suppose one  changes $\beta$ in such way that $\beta$ increases converging to $\infty$, then one can ask about  the asymptotic  behavior of the stationary Gibbs probability
 $\mu_{\beta}$. One should analyze first what that happens with the optimal $\phi$ (or almost optimal) in the maximization problem above.
In order to answer this last question, we use,
in  $\bb L^2(\mu_{\beta})$,  the Schwartz inequality, and we obtain
$$|\<\phi^{1/2} ,\, \,\mathcal{L}_\beta (\phi^{1/2})\>_{\mu_{\beta}}| \,\leq\,
 \Vert\phi\Vert_2\,\, \Vert\mathcal{L}_\beta (\phi^{1/2})\Vert_2\leq d \Vert\phi\Vert_2=d .$$
Note that, for a fixed large $\beta$, the positive value  $\gamma_\beta(x)=1-\beta V(x)+ \lambda_{\beta V}$ became smaller close by the supremum of $V$.
Which means that $\frac{1}{\gamma_\beta(x)}$ became large close by the supremum of $V$.
Moreover, for fixed $\beta$, the part $\, \int
 \beta V |\phi|\,\,\, \frac{1}{ \gamma_\beta}\,
\frac{\mbox{d}\mu_{B_{\beta}}}{\int \frac{1}{\gamma_\beta}\, \mbox{d} \mu_{B_{\beta}}}\,$ of the above expression  increase if we consider $|\phi|$ such that the big part of its mass is more and more close by to the supremum of $\beta V$.
Note that, for fixed $\beta$, the part $\, \int
\{ \phi^{1/2}  \mathcal{L}_\beta (\phi^{1/2})\,-1\}
\frac{\mbox{d}\mu_{B_{\beta}}}{\int \frac{1}{\gamma_\beta}\, \mbox{d} \mu_{B_{\beta}}}\,$ of the above expression is bounded and just depends on $\phi$.
The supremum of  $\, \int
 \beta V |\phi|\,\,\, \frac{1}{ \gamma_\beta}\,
\frac{\mbox{d}\mu_{B_{\beta}}}{\int \frac{1}{\gamma_\beta}\, \mbox{d} \mu_{B_{\beta}}}\,$ grows with $\beta$  at least of order $\beta$.

Therefore, for large $\beta$,
the maximization above should be obtained by taking $\phi=\phi_\beta$ in $\bb L^2(\mu_{\beta})$
such that is more and more concentrated close by the supremum of $\beta V$.
In this way, when $\beta \to \infty$ the "almost" optimal $\phi$ has a tendency to localize the points where the supremum of $V$ is attained.
If there is a unique point $z_0$ where $V$ is optimal, then $\lambda_\beta \sim \beta V(z_0)$. The probability $\mu_\beta$ will converge to the
delta Dirac on the point $z_0.$  This procedure is quite similar with  the process of determining ground states for a given potential via
an approximation by Gibbs states which have a very small value of temperature (see for instance \cite{BLL}).

The Metropolis algorithm has several distinct applications.  In one of them,  it can be used to maximize a function on a quite large space
(see  \cite{DS2}  and \cite{Leb}). Suppose $V$ has a unique point of maximal value.
The basic idea is to produce a random algorithm that can explore the state space  and localize the point of maximum, this problem may happen with a deterministic algorithm.
The use of continuous time paths resulted in some advantages in the method.
The randomness assures that the algorithm does note stuck on a point of local maximum of some function $V$. The setting we consider here has several similarities with the
usual procedure. When we take $\beta$ large, then  the probability $\mu_\beta$ will be very close to the
delta Dirac on the point of maximum for $V$ as we just saw. This is so because  the parameter $\frac{1}{\gamma_\beta(x)}$
of the exponential distribution became large close by the supremum of $V$.  In the classical Metropolis algorithm there is link on $\beta $ and $t$ which is
necessary for the convergence (cooling schedule in \cite{S1}). In a forthcoming paper, using our large deviation results, we will investigate the question:
given small  $\epsilon$ and $\delta$, with probability bigger than $1-\delta$, the empirical path on the one-dimensional spin lattice will
stay, up to a distance smaller the $\epsilon$ of the maximal value, a proportion $1-\delta$ of the time $t$, if $t$ and $\beta $ are chosen in a certain way (to be understood).
In order to do that we have to use the large deviation results we get before.

\section{Ergodicity of the shift $\Theta_t: \mc D\to\mc D$ relative to $\bb P_{\mu_{A}}$}

The probability $\bb P_{\mu_{A}}$ was obtained from $\{X_t=X_t^{\mu_A}\!\!,\;\,\, t \geq 0\}$.

This section is devoted to show the ergodicity for the continuous time shift $\Theta_t: \mc D\to\mc D$, when we have that the limit below exists:
$$ \lim_{t\to\infty}\frac{1}{t}\int_0^t P_{s}(F)(x)\,ds=\int F\,d\mu_A.$$
The ideas presented here are based in \cite{L1} and \cite{L2}.

Consider $f,g$  functions of $n$ variables. For all $0\leq t_1<\dots< t_n$ define the functions
$F$ and $G$ in one variable by
$$F(x)= \bb E_{x}\big[f(X_{0},X_{t_2-t_1}\dots,X_{t_n-t_1})\big]$$
$$G(X_{t_n})= \bb E_{\mu_A}\big[g(X_{t_1},\dots,X_{t_n})|X_{t_n}\big]$$
Using the Markov property, for $s>t_n-t_1$, we can write
\begin{equation*}
 \begin{split}
  \bb E_{\mu_A}\big[g(X_{t_1},\dots,X_{t_n})f(X_{t_1+s},\dots,X_{t_n+s})\big]
  &=\bb E_{\mu_A}\Big[g(X_{t_1},\dots,X_{t_n})\bb E_{\mu_A}\big[f(X_{t_1+s},\dots,X_{t_n+s})|\mc F_{t_n}\big]\Big]\\
  &=\bb E_{\mu_A}\Big[\bb E_{\mu_A}\big[g(X_{t_1},\dots,X_{t_n})|X_{t_n}\big]\bb E_{\mu_A}\big[f(X_{t_1+s},\dots,X_{t_n+s})|\mc F_{t_n}\big]\Big]\\
  &=\bb E_{\mu_A}\Big[G(X_{t_n})\bb E_{X_{t_n}}\big[f(X_{t_1+s-t_n},\dots,X_{s})\big]\Big].\\
 \end{split}
\end{equation*}
Since $\{X_t,t\geq 0\}$ is stationary, we obtain
\begin{equation*}
 \begin{split}
  \bb E_{\mu_A}\big[f(X_{t_1+s},\dots,X_{t_n+s})\big]
  &=\bb E_{\mu_A}\Big[G(X_{0})\bb E_{X_{0}}\big[f(X_{t_1+s-t_n},\dots,X_{s})\big]\Big]\\
  &=\bb E_{\mu_A}\Big[G(X_{0})\bb E_{X_{0}}\big[\bb E_{\mu_A}[f(X_{t_1+s-t_n},\dots,X_{s})|\mc F_{t_1+s-t_n}]\big]\Big].
 \end{split}
\end{equation*}
Applying again the Markov property, we get
\begin{equation*}
 \begin{split}
\bb E_{\mu_A}\Big[G(X_{0})\bb E_{X_{0}}\big[\bb E_{X_{t_1+s-t_n}}[f(X_{0},\dots,X_{X_{t_n-t_1}})]\big]\Big]&=
\bb E_{\mu_A}\Big[G(X_{0})\bb E_{X_{t_1+s-t_n}}\big[f(X_{0},\dots,X_{X_{t_n-t_1}})\big]\Big]\\
&=\bb E_{\mu_A}\big[G(X_{0})F(X_{t_1+s-t_n})\big]\\
&=\int \bb E_{x}\big[G(X_{0})F(X_{t_1+s-t_n})\big]d\mu_A(x)\\
&=\int G(x) P_{t_1+s-t_n}(F)(x)d\mu_A(x).
 \end{split}
\end{equation*}
Then, we have that
\begin{equation*}
 \begin{split}
  \lim_{t\to\infty}\frac{1}{t}\int_0^t\bb E_{\mu_A}\big[g(X_{t_1},\dots,X_{t_n})f(X_{t_1+s},\dots,X_{t_n+s})\big]\, ds&
  =\lim_{t\to\infty}\int G(x)\frac{1}{t}\int_0^t P_{t_1+s-t_n}(F)(x)\,ds\,d\mu_A(x)
 \end{split}
\end{equation*}
The limit
$$ \lim_{t\to\infty}\frac{1}{t}\int_0^t P_{t_1+s-t_n}(F)(x)\,ds$$
exists and it is equal to $\int F\, d\mu_A$ (see o beginning of the Section 2).
Thus,
\begin{equation*}
 \begin{split}
  \lim_{t\to\infty}\frac{1}{t}\int_0^t\bb E_{\mu_A}\big[g(X_{t_1},\dots,X_{t_n})f(X_{t_1+s},\dots,X_{t_n+s})\big]\, ds&
  =\int G(x)\,d\mu_A(x)\,\int F(x)\,d\mu_A(x)\\
  &= \bb E_{\mu_A}\big[g(X_{t_1},\dots,X_{t_n})\big] \bb E_{\mu_A}\big[f(X_{t_1},\dots,X_{t_n})\big]\,.
 \end{split}
\end{equation*}
Now, consider $f$ such that $f(\Theta_s\circ(X_{t_1}(w),\dots,X_{t_n})(w))=f(X_{t_1}(w),\dots,X_{t_n}(w))$, for all $s$ and $w$, then
\begin{equation*}
 \begin{split}
  E_{\mu_A}\big[g(X_{t_1},\dots,X_{t_n})f(X_{t_1},\dots,X_{t_n})\big]\, &=
  \lim_{t\to\infty}\frac{1}{t}\int_0^t\bb E_{\mu_A}\big[g(X_{t_1},\dots,X_{t_n})f(X_{t_1},\dots,X_{t_n})\big]\, ds\\
  & =\lim_{t\to\infty}\frac{1}{t}\int_0^t\bb E_{\mu_A}\big[g(X_{t_1},\dots,X_{t_n})f(\Theta_s\circ(X_{t_1},\dots,X_{t_n}))\big]\, ds\\
  &=\lim_{t\to\infty}\frac{1}{t}\int_0^t\bb E_{\mu_A}\big[g(X_{t_1},\dots,X_{t_n})f(X_{t_1+s},\dots,X_{t_n+s})\big]\, ds\\
  &= \bb E_{\mu_A}\big[g(X_{t_1},\dots,X_{t_n})\big] \bb E_{\mu_A}\big[f(X_{t_1},\dots,X_{t_n})\big]\,.
 \end{split}
\end{equation*}
Take $g$ equal to $f$, then
\begin{equation*}
 \begin{split}
  E_{\mu_A}\big[f(X_{t_1},\dots,X_{t_n})^2\big]\,
  &=  \bb E_{\mu_A}\big[f(X_{t_1},\dots,X_{t_n})\big]^2\,.
 \end{split}
\end{equation*}
The last equality implies that $f$ is constant (almost surely).

Considering that $\{X_t, t\geq 0\}$ is the canonical process, i.e.,  $X_t(w)=w_t$, for all $t\geq 0$, and $w\in \mc D$,
we can rewritten our result as:
given a function
$$ w\in \mc D\mapsto f(w_{t_1},\dots,w_{t_n}),$$
which is invariant for the continuous time shift $\Theta_s:\mc D\to\mc D$, we get that this function is constant.

Note that all mensurable function $H:\mc D\to \bb R$ depends on a countable set of coordinates.
Then, without loss of generality, suppose that $H(w)=h(w_{s_1},\dots,w_{s_n},\dots)$, where $h:\Big(\{1,\dots,d\}^{\bb N}\Big)^{\bb N}\to \bb R$.
Therefore, using  approximation arguments one also get that $H$ is constant.

\end{document}